\documentclass[a4paper]{article}

\usepackage{authblk}
\usepackage{xspace}
\usepackage{mathtools,dsfont}
\usepackage{amsmath,amssymb,amsthm}

\usepackage{multirow}
\usepackage{diagbox}
\usepackage{todonotes}

\usepackage[backend=biber,maxbibnames=4]{biblatex}
\usepackage[hidelinks]{hyperref}

\usepackage[algo2e,ruled,vlined,commentsnumbered]{algorithm2e}

\usepackage{tikz}
\usetikzlibrary{matrix, positioning}
\usepackage{pgfplots}
\usepackage{tikzscale}
\usepackage{adjustbox}
\usepackage{subcaption}
\usepackage[group-digits=integer,group-minimum-digits=4,print-unity-mantissa=false]{siunitx}
\usepackage{booktabs}

\newcommand{\ie}{i.e.,\xspace}

\newcommand{\Real}{\mathbb{R}}
\newcommand{\Nat}{\mathbb{N}}

\newcommand{\numBlocks}{N}
\newcommand{\numArcs}{M}

\newcommand{\blockSize}{n}

\newcommand{\numCoupled}{l}
\newcommand{\numOutCoupled}[1]{\numCoupled^{+}_{#1}}
\newcommand{\numInCoupled}[1]{\numCoupled^{-}_{#1}}

\newcommand{\Graph}{\mathbb{D}}
\newcommand{\Vertices}{\mathbb{V}}
\newcommand{\Arcs}{\mathbb{A}}
\newcommand{\rootVert}{R}

\newcommand{\identMat}{\mathbb{I}}
\newcommand{\zeroMat}{0}

\newcommand{\incArcs}[1]{\delta^{-}(#1)}
\newcommand{\outArcs}[1]{\delta^{+}(#1)}

\newcommand{\incMat}[1]{C^{-}_{#1}}
\newcommand{\outMat}[1]{C^{+}_{#1}}

\newcommand{\define}{\coloneqq}
\newcommand{\enifed}{\eqqcolon}

\DeclareMathOperator{\height}{height}
\DeclareMathOperator{\depth}{depth}

\newcommand{\heightGraph}{\height(\Graph)}

\DeclareMathOperator{\head}{head}
\DeclareMathOperator{\tail}{tail}

\DeclareMathOperator{\even}{even}
\DeclareMathOperator{\odd}{odd}

\DeclareMathOperator{\blkdiag}{blkdiag}
\DeclareMathOperator{\diag}{diag}

\DeclareMathOperator{\super}{super}

\DeclareMathOperator{\hook}{hook}
\DeclareMathOperator{\exact}{exact}
\DeclareMathOperator{\recursive}{rec}

\newcommand{\Schur}{\mathcal{S}}
\newcommand{\Precond}{\mathcal{P}}

\newcommand{\MLSmoother}{G}
\newcommand{\MLProlong}{P}
\newcommand{\MLOperator}{\Psi}
\newcommand{\MLNumLevels}{K}
\newcommand{\MLDimension}{n}

\DeclareMathOperator{\MLSubset}{S}

\theoremstyle{plain}
\newtheorem{thm}{Theorem}[section]

\newtheorem{lem}[thm]{Lemma}

\newtheorem{cor}[thm]{Corollary}

\theoremstyle{definition}
\newtheorem{ass}{Assumption}

\renewcommand{\SetProgSty}[1]{\renewcommand{\ProgSty}[1]{\textnormal{\csname#1\endcsname{##1}}\unskip}}%

\SetKwComment{Comment}{$\triangleright$\ }{}
\SetKwInOut{Input}{Input}
\SetKwInOut{Output}{Output}
\SetKwFor{Until}{until}{}

\SetKwProg{Fn}{}{}{}

\SetProgSty{textsc}
\SetFuncArgSty{}
\SetFuncSty{textsc}

\pgfdeclarelayer{background layer}
\pgfdeclarelayer{foreground layer}
\pgfsetlayers{background layer,main,foreground layer}

\tikzset{layer/.style={%
    execute at begin scope={\pgfonlayer{#1}},
    execute at end scope={\endpgfonlayer}
}}

\graphicspath{{Figures/}}

\newcommand{\algoabovedisplayskip}{2pt}
\newcommand{\algobelowdisplayskip}{2pt}

\newlength{\savedabovedisplayskip}
\newlength{\savedbelowdisplayskip}

\newenvironment{AlgoEq}[0]{
  \setlength{\savedabovedisplayskip}{\abovedisplayskip}
  \setlength{\savedbelowdisplayskip}{\belowdisplayskip}
  \setlength{\abovedisplayskip}{\algoabovedisplayskip}
  \setlength{\belowdisplayskip}{\algobelowdisplayskip}
  }{
  \setlength{\abovedisplayskip}{\savedabovedisplayskip}
  \setlength{\belowdisplayskip}{\savedbelowdisplayskip}
}

\newlength\figureheight
\newlength\figurewidth
\definecolor{plotBlue}{HTML}{1f77b4}
\definecolor{plotOrange}{HTML}{ff7f0e}
\definecolor{plotGreen}{HTML}{2ca02c}
\definecolor{plotRed}{HTML}{d62728}

\newcommand{\thetitle}{A Framework for the Solution of Tree-Coupled Saddle-Point Systems}

\newcommand{\firstauthor}{Christoph Hansknecht}
\newcommand{\secondauthor}{Bernhard Heinzelreiter}
\newcommand{\thirdauthor}{John W. Pearson}
\newcommand{\fourthauthor}{Andreas Potschka}

\hypersetup{
  pdftitle = {\thetitle},
  pdfauthor = {\firstauthor, \secondauthor, \thirdauthor, \fourthauthor}
}

\newlength{\stextwidth}
\newcommand\makesamewidth[3][c]{
  \settowidth{\stextwidth}{#2}
  \makebox[\stextwidth][#1]{#3}
}

\newcommand{\innerOp}[1]{#1^{\circ}}
\newcommand{\innerGraph}{\innerOp{\Graph}}
\newcommand{\innerVertices}{\innerOp{\Vertices}}
\newcommand{\innerArcs}{\innerOp{\Arcs}}

\begin{document}

\title{\thetitle}

\author[1]{\firstauthor}
\author[2]{\secondauthor}
\author[2]{\thirdauthor}
\author[1]{\fourthauthor}

\affil[1]{\small Institute of Mathematics, Clausthal University of Technology, Erzstr. 1, 38678 Clausthal-Zellerfeld, Germany}
\affil[2]{\small School of Mathematics, The University of Edinburgh, James Clerk Maxwell Building, The King’s
  Buildings, Peter Guthrie Tait Road, Edinburgh, EH9 3FD, United Kingdom}

\date{\today}

\maketitle

\begin{abstract}
  We consider the solution of saddle-point systems with a
  tree-based block structure, introducing a parallelizable direct
  method for their solution. As our key contribution, we then propose several
  structure-exploiting preconditioners to be used during applications
  of the MINRES and GMRES algorithms and analyze their properties. We
  adapt several concepts originating in the field of
  multigrid methods, obtaining a variety of problem-adapted
  multi-level methods. We analyze the complexity of all algorithms, and derive a number of results on eigenvalues of the preconditioned system and convergence of iterative methods. We
  validate our theoretical findings through a range of numerical experiments.
\end{abstract}

\section{Introduction}

The numerical solution of (generalized) saddle-point systems of the type
\begin{equation}
  \label{eq:2by2_coupled_system}
  \begin{pmatrix}
    \mathcal{B} & \mathcal{C}^{T} \\
    \mathcal{C} & -\mathcal{D}
  \end{pmatrix}
  \begin{pmatrix}
    x \\
    y
  \end{pmatrix}
  =
  \begin{pmatrix}
    h \\
    f
  \end{pmatrix}
\end{equation}
has been studied extensively, due to their wide applicability to a
range of fields (see~\cite{numerical_saddle_point} for a survey).
Saddle-point problems are called \emph{normal} when $\mathcal{D} = \zeroMat$
and \emph{generalized} otherwise,
and  appear, for example, as subproblems in many
optimization methods such as
\emph{sequential quadratic programming}~\cite{sl1qp},~\cite[Sec. 12.4]{practical_methods},
\emph{interior point}~\cite[Ch. 19]{numerical_optimization}, and
\emph{sequential homotopy}~\cite{sequential_homotopy_precond}
methods as well as in the context of partial
differential equations (PDEs), following discretizations using
\emph{mixed finite element methods}~\cite{mixed_fem}, a notable
example being the discretized Stokes equation.
Consequently, suitable
preconditioners for such systems have been thoroughly examined
in the context of PDE discretizations
(see~\cite[Ch. 4]{finite_elements} for a summary).
Block-diagonal~\cite{fast_and_robust,black_box_precond} and
block-triangular~\cite{block_triangular_preconds} preconditioners have
been proposed and examined with respect to the spectra of the respective
preconditioned systems, with the overarching goal of bounding their range
independently of discretization parameters.
In more general problem settings, the class of \emph{constraint preconditioners} arises by
approximating $\mathcal{B}$ with another matrix $\mathcal{G}$ in a preconditioner
$\left( \begin{smallmatrix} \mathcal{G} & \mathcal{C}^{T} \\ \mathcal{C} & - \mathcal{D} \end{smallmatrix}  \right)$.
These preconditioners were first~\cite{constraint_preconds} examined for normal and
subsequently~\cite{regularized_constraint_preconds} generalized saddle-point systems, establishing
spectral properties. In the area of nonlinear programming, efforts have also been
made~\cite{primal_block_angular, parallel_interior} in order to exploit
specific block-structures of saddle-point systems in the course of their
numerical solution. More recently, the analysis of preconditioners has also
been extended to both double~\cite{sequential_homotopy_precond,double_saddle_point_spectral}
and multiple~\cite{multi_saddle_point} saddle-point systems.

A particular source of well-structured linear systems arises in stochastic programming
problems~\cite{stochastic_programming}, where
different scenarios are often largely but not completely independent. For general convex
quadratic programs, a framework for interior-point methods named
\texttt{OOPS} has been proposed~\cite{oops} and
\emph{tree-sparse} quadratic problems have been studied
extensively~\cite{tree_sparse,recursive_direct}. For time-dependent
PDEs, \emph{parallel-in-time} methods such as
\texttt{PITA}~\cite{PITA} and \texttt{PFASST}~\cite{PFASST} are
effective tools which similarly exploit coupling structures in order to achieve peak
utilization of massively parallel processors, augmenting other
well-known methods such as
the \emph{overlapping Schwarz method}~\cite{algebraic_schwarz} for domain decomposition in space,
which itself has been generalized to solve \emph{graph-based}
quadratic programs~\cite{decentralized_schemes}. Indeed, interfaces to these
problem-specific solvers are combined within a package called
\texttt{Plasmo.jl}~\cite{plasmo}, allowing for the generic inclusion of
network information at the modeling stage of nonlinear problems.

In this paper we derive a suite of direct and (in particular) iterative solvers for saddle-point systems with a \emph{tree-coupled} structure. Specifically, we extend previous structure-exploiting approaches for saddle-point
systems by incorporating a graph-based coupling structure, where
interactions between individual and otherwise isolated subsystems are
expressed via generic coupling constraints.
To this end, let $\Graph=(\Vertices, \Arcs)$ be a
directed tree (an arborescence), with $\numBlocks$ vertices
$\Vertices \define \{1, \ldots, \numBlocks\}$ and $\numArcs$ arcs
$\Arcs \define \{a_1, \ldots, a_{\numArcs}\} \subseteq \Vertices
\times \Vertices$ directed away from a root $\rootVert \in
\Vertices$. Each vertex has associated variables
$x_{i} \in \Real^{\blockSize_{i}}$, for $\blockSize_{i} \in \Nat$, which
are coupled along the arcs in $\Arcs$. For each arc $a_{k} = (i, j)$,
two matrices
$\outMat{k} \in \Real^{\numCoupled_k \times \blockSize_{i}}$ and
$\incMat{k} \in \Real^{\numCoupled_k \times \blockSize_j}$ describe
the coupling between variables $x_{i}$ and $x_{j}$. Specifically, if
we let $\outArcs{i}$ and $\incArcs{i}$ denote the outgoing and
incoming arcs of $i \in \Vertices$, respectively, the saddle-point
system we shall investigate is defined by
\begin{equation}
  \label{eq:tree_coupled_system}
  \begin{aligned}
    B_{i} x_{i} + \sum_{\mathclap{a_k \in \outArcs{i}}} (\outMat{k})^{T} y_k - \sum_{\mathclap{a_k \in \incArcs{i}}} (\incMat{k})^{T} y_k = h_{i} & \quad \text{for all } i \in \Vertices, \\
    \outMat{k} x_{i} - \incMat{k} x_{j} - D_k y_k = f_k & \quad \text{for all } a_k = (i, j) \in \Arcs,
  \end{aligned}
\end{equation}
where
$B_{i} = B_{i}^{T} \in \Real^{\blockSize_{i} \times \blockSize_{i}}$
corresponds to conditions on $x_{i}$ with a right-hand side
$h_{i} \in \Real^{\blockSize_{i}}$, and
$y_{k} \in \Real^{\numCoupled_{k}}$ are coupling variables with
conditions given by matrices
$D_{k} = D_{k}^{T} \in \Real^{\numCoupled_{k} \times \numCoupled_{k}}$
and right-hand sides $f_{k} \in \Real^{\numCoupled_{k}}$~(see
Figure~\ref{fig:tree_coupled} for an example). This system
is symmetric and can be seen to be a special case of
~\eqref{eq:2by2_coupled_system} by setting
\begin{equation*}
  \begin{aligned}
    \mathcal{B} &\define \blkdiag(B_1, \ldots, B_{\numBlocks}), \\
    \mathcal{D} &\define \blkdiag(D_1, \ldots, D_{\numArcs})\text{, and} \\
    \mathcal{C} &= (\mathcal{C}_{k, i}) \define
                  \begin{cases}
                    \outMat{k} & \text{ if } a_{k} = (i, j), \\
                    -\incMat{k} & \text{ if } a_{k} = (j, i), \\
                    0 & \text{ otherwise.}
                  \end{cases}
                  \quad
                  \text{ for } i \in \Vertices \text{ and } k \in \{1, \ldots, \numArcs\},
  \end{aligned}
\end{equation*}
where the notation  $\blkdiag(B_{1}, \ldots, B_{\numBlocks})$ denotes a matrix consisting
of $\numBlocks$ blocks of rows and columns with diagonal blocks set to the matrices
$B_{i}$ and off-diagonal blocks set to zero matrices of appropriate dimensions (and similarly for other uses of the `$\blkdiag$' notation).
Systems of the form~\eqref{eq:tree_coupled_system}
arise from a broader class of problems. Specifically, consider a nonlinear programming problem
with a separable objective function $\sum_{i \in \Vertices} \phi_{i}(\zeta_{i})$, and constraints
of the form
\begin{equation*}
  \begin{aligned}
    c_{i}(\zeta_{i}) &= 0 \quad \forall i \in \Vertices && | \: \cdot \nu_{i} \\
    \outMat{k} \zeta_{i} - \incMat{k} \zeta_{j} &= 0 \quad \forall a_{k} = (i, j) \in \Arcs && | \: \cdot y_{k} \\
  \end{aligned}
\end{equation*}
composed of (possibly nonlinear) internal constraints as well as
linear coupling constraints on the graph $\Graph$, with corresponding
Lagrange multipliers $\nu_{i}$ and $y_{k}$, respectively.
If this problem has a quadratic objective
$\phi_{i}(\zeta_{i}) = g_{i}^{T} \zeta_{i} + \tfrac{1}{2} \zeta_{i}^{T} H_{i} \zeta_{i}$ and linear constraints
$c_{i}(\zeta_{i}) \define A_{i} \zeta_{i} - b_{i} \; (\stackrel{!}{=} 0)$, its Karush--Kuhn--Tucker (KKT) system is
of the form~\eqref{eq:tree_coupled_system}, where
\begin{equation*}
  B_{i} =
  \begin{pmatrix}
    H_{i} & A_{i}^{T} \\
    A_{i} & 0
  \end{pmatrix}
  ,~\
  x_{i} =
  \begin{pmatrix}
    \zeta_{i} \\
    \nu_{i}
  \end{pmatrix}
  ,~\
  h_{i} =
  \begin{pmatrix}
    - g_{i} \\
    b_{i}
  \end{pmatrix}
  ,~\
  D_{k} = 0
  ,~\
  f_{k} = 0.
\end{equation*}
For general nonlinear $\phi_{i}$ and $c_{i}$, we can employ
an interior point method~\cite[Ch. 19]{numerical_optimization}, generating a sequence $(\zeta^{(l)}_{i}, \nu^{(l)}_{i}, y^{(l)}_{k})_{l}$
of primal--dual solutions based on a given starting point. At each iteration,
a system of the form~\eqref{eq:tree_coupled_system} is solved, where the internal
systems have a matrix $H_i$ corresponding to
the Hessian of the Lagrangian $\phi_{i} + \nu_{i}^{T} c_{i}$ plus a barrier term, and
$A_i$ to the Jacobian of $c_i$ evaluated at the current primal--dual iterate, with $D_{k}$ diagonal.
An alternative to an interior point method is the sequential homotopy
method~\cite{sequential_homotopy,sequential_homotopy_precond}, which also uses linear systems as an
algorithmic backbone. The linear systems to be solved are again
of the form~\eqref{eq:tree_coupled_system} with blocks given by
\begin{equation*}
  B_i =
  \begin{pmatrix}
    H_{i} + \lambda \identMat & A_{i}^{T} \\
    A_{i} & - \delta \identMat
  \end{pmatrix}
  \quad \text{ and }
  \quad
  D_{k} = \delta \identMat,
\end{equation*}
where $\identMat$ denotes the identity matrix of approximate
dimension and $\lambda, \delta > 0$ are algorithmic parameters.

Lastly, note that we make no specific assumptions
regarding the coupling matrices $\outMat{k}$ and $\incMat{k}$, as our
algorithmic framework does not require any such assumption. In terms
of modeling, a common choice for these matrices stems from the
enforcement of \emph{consensus constraints}, \ie requiring certain
entries of the variables $x_{i}$ and $x_{j}$ to
coincide. For this particular case of coupling,
$\outMat{k}$ and $\incMat{k}$ then consist of rows of
positive or negative unit vectors.

In this paper, we provide a new mathematical framework for deriving and analysing direct and preconditioned iterative methods for 
such tree-coupled systems. We propose a range of solution algorithms and implement them in a 
purely sequential fashion; we highlight that these methods are designed to be amenable to parallelization, however this would require a bespoke implementation, so we apply our methods sequentially in order to focus on the linear algebra aspects in this work. Aside from a parallelizable direct method, we implement a range of structured preconditioners which may be embedded within suitable Krylov subspace methods, including block preconditioners, recursive preconditioners, and multi-level approaches. We prove a range of results relating to the convergence, complexity, and spectral properties of our algorithms. Finally, we apply our methodology to problems from a number of fields, including model predictive control, multiple shooting for optimal control, and domain decomposition. These results validate our theoretical results and demonstrate the versatility of our mathematical approach.

\subsection{Notation and Definitions}

A vertex $i \in \Vertices$ is said to be a \emph{leaf} of the tree
$\Graph$ iff $\outArcs{i} = \emptyset$ and an \emph{inner vertex}
otherwise. The \emph{inner subgraph}, denoted by $\innerGraph = (\innerVertices, \innerArcs)$
is the subgraph induced by the set $\innerVertices$ of inner vertices.
We generally assume that the sets $\outArcs{i}$ and
$\incArcs{i}$ are ordered consistently and let
$\delta(i) \define \outArcs{i} \cup \incArcs{i}$.
For each arc $a_{k} = (i, j) \in \Arcs$ we set
$\head(a_{k}) \define j$ and $\tail(a_{k}) \define i$.
The \emph{parent} of a vertex $i \in \Vertices$, $i \neq \rootVert$ is
the vertex $k \in \Vertices$ such that $(k, i) \in \Arcs$ and $k_{i}$
denotes the index of the arc entering $i$, \ie $k_{i} \in \{1, \ldots, \numArcs\}$ is such
that $\incArcs{i} = \{a_{k_i} \}$ and $a_{k_i} = (k, i)$.

For each vertex $i \in \Vertices$ the \emph{children} of $i$
are the head vertices of the arcs in $\outArcs{i}$.
The \emph{depth} of $i$, which we denote as $\depth(i)$, is defined as
the length of the (unique) $(\rootVert, i)$-path in $\Graph$.  Similarly, the
\emph{height} of $i$, denoted $\height(i)$, is defined to be zero if
$i$ is a leaf, and the maximum height of any child vertex in $\outArcs{i}$
plus one otherwise.  The height of $\Graph$, $\heightGraph$, is
defined as the height of $\rootVert$ or, equivalently, as the maximum depth of
any vertex in $\Vertices$.
The subtree rooted at vertex $i$ is denoted by
$\Graph_{\leq i} = (\Vertices_{\leq i}, \Arcs_{\leq i})$ and given by
the union of the vertices and arcs on all $(i, j)$-paths in $\Graph$.
We also let $\numOutCoupled{i} \define \sum_{a \in \outArcs{i}} \numCoupled_{k}$,
$\numInCoupled{i} \define \sum_{a \in \incArcs{i}} \numCoupled_{k}$, and
$\numCoupled_{i} \define \numOutCoupled{i} + \numInCoupled{i}$ be
the outgoing, incoming, and total number of variables coupled to $i \in \Vertices$ respectively.
An example for these definitions is given
in Figure~\ref{fig:tree_coupled}.

We also use lower case letters to denote vectors, upper case ones
for provided matrices, and curly upper case letters for larger block matrices.
Lastly, we present results regarding complexity in the usual
$\mathcal{O}$-notation~\cite[Ch. 1]{combopt}, where for
functions $f, g : \Nat \to \Nat$ we say that $f \in \mathcal{O}(g)$ if
$\limsup_{n \to \infty} f(n) / g(n) < \infty$ and $f \in \Theta(g)$ if
$f \in \mathcal{O}(g)$ and $g \in \mathcal{O}(f)$.

\subsection{Assumptions}
Besides the symmetry of the matrices $B_i$ and $D_{k}$, we make the following
additional assumption to ensure the non-singularity of \eqref{eq:2by2_coupled_system}:
\begin{ass}
  \label{ass:general_assumptions}
  \begin{enumerate}
  \item
    The matrix $\mathcal{B}$ is invertible.
  \item
    The Schur complement of \eqref{eq:2by2_coupled_system}, given by
    \begin{equation}
      \label{eq:2by2_schur_complement}
      \mathcal{S} \define \mathcal{C} \mathcal{B}^{-1} \mathcal{C}^{T} + \mathcal{D},
    \end{equation}
    is positive definite.
  \end{enumerate}
\end{ass}
\begin{lem}
  \label{lem:system_invertible}
  Under Assumption~\ref{ass:general_assumptions}, system~\eqref{eq:2by2_coupled_system}
  is invertible.
\end{lem}
\begin{proof}
  Saddle-point systems of the form~\eqref{eq:2by2_coupled_system} can
  be decomposed~\cite[Eq. (3.1)]{numerical_saddle_point} into the product
  \begin{equation*}
    \begin{pmatrix}
      \mathcal{B} & \mathcal{C}^{T} \\
      \mathcal{C} & -\mathcal{D}
    \end{pmatrix} =
    \begin{pmatrix}
      \identMat & \zeroMat \\
      \mathcal{C} \mathcal{B}^{-1} & \identMat
    \end{pmatrix}
    \begin{pmatrix}
      \mathcal{B} & \zeroMat \\
      \zeroMat & -\mathcal{S}
    \end{pmatrix}
    \begin{pmatrix}
      \identMat & \mathcal{B}^{-1} \mathcal{C}^{T} \\
      \zeroMat & \identMat
    \end{pmatrix}.
  \end{equation*}
  The matrices in this product are all invertible since $\mathcal{B}$ is invertible and $\mathcal{S}$
  is positive definite.
\end{proof}
Regarding Assumption~\ref{ass:general_assumptions} it is apparent from
Lemma~\ref{lem:system_invertible} that non-singularity of $\Schur$ is
sufficient to ensure that system~\eqref{eq:2by2_coupled_system} is
invertible. We will however rely on
positive definiteness of $\Schur$ in particular in
Section~\ref{sec:direct_preconds}.
While we have verified that this stronger assumption
is satisfied for a number of problems, our numerical
experiments indicate that our methods work well even if $\Schur$ is
merely invertible.
\begin{figure}[ht]
  \centering
  \begin{subfigure}{.2\textwidth}
    \centering
    \includegraphics{tree_coupled}
  \end{subfigure}
  \begin{subfigure}{.79\textwidth}
    \centering
    \includegraphics{tree_system}
  \end{subfigure}
  \caption{
    Example of a tree-coupled system, based on a tree with $\numBlocks = 3$ vertices $\Vertices = \{1, 2, 3\}$ and
    $\numArcs = 2$ arcs $\Arcs = \{a_{1} = (3, 1), a_{2} = (3,2)\}$.
    Vertices $1$ and $2$ are leaves each having a height of zero, a depth of one, and $3$ as their parent.
    Vertex $\rootVert = 3$ is an inner vertex with a height of one (equal to the height of $\Graph$), a depth of zero,
    and $1$ and $2$ as its children. The inner subgraph $\innerGraph$ consists of vertex $3$ without containing any arcs.
    The arcs entering $1$ and $2$ are $a_1$ and $a_2$ respectively, \ie
    it holds that $k_{1} = 1$ and $k_{2} = 2$. Both arcs have $3$ as their tail while $\head(a_1) = 1$ and $\head(a_2) = 2$.
    The subtree rooted at $3$ is equal to the graph itself, whereas $\Graph_{\leq 1}$ and $\Graph_{\leq 2}$ consist
    of only the vertices $1$ and $2$ respectively without any arcs.
  }
  \label{fig:tree_coupled}
\end{figure}

\section{Direct Method}
\label{sec:direct_method}

In order to solve the system~\eqref{eq:tree_coupled_system} we make
use of a Schur complement approach rather than a complete sparse decomposition,
which has two advantages:
First, a decomposition may be unnecessary in particularly if only a
few variables are coupled, \ie
$\numCoupled_{k} \ll \min(\blockSize_i, \blockSize_j)$. In this case
the systems involving the matrices $B_{i}$ are largely independent,
the corresponding Schur complements are small in size, and a substantial portion
of the computations may be carried out in parallel in order to improve performance
and scale to larger systems.
Second, our approach is highly flexible in how
systems involving the matrices $B_{i}$ are solved. Thus, any
structure-exploiting solution methods for solving these systems can be
easily incorporated into our computational framework.

\subsection{Structure of Algorithm}
We begin by giving a direct method (see Algorithm~\ref{alg:arrowhead_direct_method})
inspired by the exploitation of a
`symmetric bordered block-diagonal structure' introduced by
\citeauthor{parallel_interior} in \cite{parallel_interior}. Using a
symmetric permutation of the blocks constituting
system~\eqref{eq:2by2_coupled_system}, we obtain a nested sequence of
systems with this exploitable structure for each $i \in \Vertices$.
The system associated with $i$ then depends recursively on all
children of $i$, thereby corresponding to the submatrix of~\eqref{eq:2by2_coupled_system}
associated with the subtree of $\Graph$ rooted at $i$.

If $i$ is a leaf, we let $\mathcal{B}_{\leq i} \define B_i$,
$\mathcal{C}^{-}_i \define \incMat{k_i}$, and
$h_{\leq i} \define h_i$. Otherwise, we let
$\outArcs{i} = (a_{k_{1}} = (i, j_{1}), \ldots, a_{k_{r_{i}}} = (i, j_{r_{i}}))$ be
the outgoing arcs of vertex $i$ and set
\begin{align}
    x^{T}_{\leq i} & \define (x^{T}_{\leq j_{1}}, \ldots, x^{T}_{\leq j_{r_{i}}}, x^{T}_{i}, y^{T}_{j_{1}}, \ldots, y^{T}_{j_{r_{i}}}), \notag \\
    h^{T}_{\leq i} & \define (h^{T}_{\leq j_{1}}, \ldots, h^{T}_{\leq j_{r_{i}}}, h^{T}_i, f^{T}_{j_{1}}, \ldots, f^{T}_{j_{r_{i}}}), \notag \\
    \mathcal{D}_{\leq i} &\define \blkdiag \left( D_{j_{1}}, \ldots, D_{j_{r_{i}}} \right), \notag \\
    \mathcal{C}_{\leq i}^{+} & \define \left((\outMat{k_{1}})^{T}, \ldots, (\outMat{k_{r_{i}}})^{T}\right)^{T},\notag \\
    \mathcal{C}_{i}^{-} & \define
                          \adjustbox{valign=c}{
                          \begin{tikzpicture}[inner sep=0pt, minimum width=.5cm]
                            \matrix (S) [every node/.style={minimum width=1.3cm, minimum height=.6cm, inner sep=0cm},
                            matrix of math nodes,
                            left delimiter=(,
                            right delimiter=),
                            ampersand replacement=\&]
                            {
                              0 \& \cdots \& 0 \& \incMat{k_i} \& 0 \\
                            };
                          \end{tikzpicture}
                          }
                          \text{ for } i \neq \rootVert \text{, and} \notag \\
  \mathcal{C}_{\leq i}^{-} & \define \blkdiag \left( \mathcal{C}_{j_{1}}^{-}, \ldots, \mathcal{C}_{j_{r_{i}}}^{-}  \right), \notag \\
  \intertext{
  where $a_{k_i}$ is the unique arc in $\incArcs{i}$ as defined above
  and the zero matrices in the definition of $\mathcal{C}_{i}^{-}$ are
  of appropriate dimensions to align the columns of $\incMat{k_i}$ with those of $B_i$ in the following
  definition of
  }
    \mathcal{B}_{\leq i} & \define
                           \adjustbox{valign=c}{
                           \begin{tikzpicture}[inner sep=0pt, minimum width=.5cm]
                             \begin{scope}[layer=main]
                               \matrix (R) [every node/.style={minimum width=1.3cm, minimum height=.7cm, inner sep=0cm},
                               matrix of math nodes,
                               left delimiter=(,
                               right delimiter=),
                               ampersand replacement=\&]
                               {
                                 |(block-begin)|\mathcal{B}_{\leq j_{1}} \& \& \& \& |(CT-begin)| \\
                                 \& \ddots \& \& \& - (\mathcal{C}_{\leq i}^{-})^{T} \\
                                 \& \& \mathcal{B}_{\leq j_{r_{i}}} \& \& |(CT-end)| \\
                                 \& \& \& |(block-end)| B_i \& (\mathcal{C}_{\leq i}^{+})^{T} \\
                                 |(C-begin)| \& |(lower)| -\mathcal{C}_{\leq i}^{-} \& |(C-end)| \& \mathcal{C}_{\leq i}^{+} \& - \mathcal{D}_{\leq i} \\
                               };
                             \end{scope}
                             \begin{scope}[layer=background layer]
                               \fill[black!5,rounded corners=0.1cm] (C-begin.north west) rectangle (C-end.south east) {};
                               \fill[black!5,rounded corners=0.1cm] (CT-begin.north west) rectangle (CT-end.south east) {};
                             \end{scope}
                             \end{tikzpicture}
                           }. \notag
\end{align}
To obtain the solution associated with the subtree rooted at $i$, we solve the
system $\mathcal{B}_{\leq i} x_{\leq i} = h_{\leq i}$. We call
the structure of $\mathcal{B}_{\leq i}$ (lower-right pointing) \emph{arrowhead} structure
rather than bordered-block diagonal as in~\cite{parallel_interior}.
Since the diagonal blocks $\mathcal{B}_{\leq j_l}$ have arrowhead
structure themselves, $\mathcal{B}_{\leq i}$ has
a \emph{nested (lower-right pointing) arrowhead} structure in general.
Since $\mathcal{B}_{\leq \rootVert}$ is a symmetric permutation
of~\eqref{eq:2by2_coupled_system}, it is invertible under
Assumption~\ref{ass:general_assumptions}.  Our approach
necessitates stronger assumptions, however. Specifically, we want to be able to
solve the nested arrowhead matrices $\mathcal{B}_{\leq i}$ using a
recursive approach based on Schur complements. To this end, we need
the following additional assumption:
\begin{ass}
  \label{ass:direct_method_assumptions}
  \begin{enumerate}
  \item
    The matrices $\mathcal{B}_{\leq i}$ are invertible.
  \item
    The Schur complements
    \begin{equation*}
      \Schur_{\leq i} \define
      \begin{pmatrix}
        -\mathcal{C}_{\leq i}^{-} & \mathcal{C}_{\leq i}^{+}
      \end{pmatrix}
      \begin{pmatrix}
        \mathcal{B}^{-1}_{\leq j_{1}} & & & \\
                                    & \ddots & & \\
                                    & & \mathcal{B}^{-1}_{\leq j_{r_{i}}} & \\
                                    & & & B^{-1}_{i}
      \end{pmatrix}
      \begin{pmatrix*}[r]
        -\left( \mathcal{C}_{\leq i}^{-} \right)^{T} \\
        \left( \mathcal{C}_{\leq i}^{+} \right)^{T}
      \end{pmatrix*}
      + \mathcal{D}_{\leq i}
    \end{equation*}
    are positive definite for all inner vertices $i \in \Vertices$.
  \end{enumerate}
\end{ass}
\begin{algorithm2e}[ht!]
  \LinesNumbered
\SetCommentSty{itshape}
\DontPrintSemicolon

\SetKwComment{Comment}{$\triangleright$\ }{}

\begin{AlgoEq}

\SetKwFunction{solveDirectSchur}{solveDirectSchur}

\Fn{solveDirectSchur($i$, $\left( \mathcal{B}_{\leq j} \right)_{j \in \Vertices_{\leq i}}$, $\left( \mathcal{S}_{\leq j} \right)_{j \in \Vertices_{\leq i}}$, $h_{\leq i}$)} {

  \Input{
    \parbox[t]{\hsize}
    {
      Vertex $i \in \Vertices$ \\
      Subtree systems $\left( \mathcal{B}_{\leq j} \right)_{j \in \Vertices_{\leq i}}$ \\
      Subtree Schur complements of $\mathcal{S}_{\leq j}$ for $j \in \Vertices_{\leq i}$ \\
      Right-hand side $h_{\leq i}$
    }
  }
  \Output{
    \parbox[t]{\hsize}
    {
      Solution $x_{\leq i} = \mathcal{B}_{\leq i}^{-1} h_{\leq i}$
    }
  }
  \If(\Comment*[f]{Base case: $x_{\leq i} = x_{i}$}){Vertex $i$ is a leaf in $\Graph$}{
    \Return $x_{i} = B_{i}^{-1} h_{i}$
    \label{alg:arrowhead_direct_method:leaf_solve}
  }
  \ForEach(\Comment*[f]{Compute right-hand sides}){$l \in \{1, \ldots, r_{i}\}$} {
    $\hat{y}_{l} \gets \solveDirectSchur{$j_l$, $\left( \mathcal{B}_{\leq k} \right)_{k \in \Vertices_{\leq j_l}}$, $\left( \mathcal{S}_{\leq k} \right)_{k \in \Vertices_{\leq j_l}}$, $h_{\leq j_l}$}$ \;
    \label{alg:arrowhead_direct_method:first_recursion}
    $\hat{x}_{l} \gets -\mathcal{C}_{l}^{-} \hat{y}_{l} + \outMat{k_l} B_{i}^{-1} h_{i} - f_{j_l}$
    \label{alg:arrowhead_direct_method:inner_first_solve}
  }
  Solve Schur complement system: $(y_{j_1},\ldots, y_{j_{r_{i}}}) \gets \mathcal{S}_{\leq i}^{-1} \left( \hat{x}_{1}, \ldots, \hat{x}_{r_{i}} \right)$ \;
  Compute solution: $x_{i} \gets B_{i}^{-1} \left( h_{i} - \sum_{l = 1}^{r_{i}} \outMat{k_l} y_{j_l} \right)$ \;
  \label{alg:arrowhead_direct_method:inner_second_solve}
  \ForEach(\Comment*[f]{Compute solutions $x_{\leq j_l}$}){$l \in \{1, \ldots, r_{i}\}$} {
    $\makesamewidth[r]{$x_{\leq j_l}$}{$\hat{z}_{l}$} \gets h_{\leq j_l} + \incMat{j_l} y_{j_l} $ \;
    $x_{\leq j_l} \gets \solveDirectSchur{$j_l$, $\left( \mathcal{B}_{\leq k} \right)_{k \in \Vertices_{\leq j_l}}$, $\left( \mathcal{S}_{\leq k} \right)_{k \in \Vertices_{\leq j_l}}$, $\hat{z}_{l}$}$
  \label{alg:arrowhead_direct_method:second_recursion}
  }
  \Return $x^{T}_{\leq i} = \left( x^{T}_{\leq j_1}, \ldots, x^{T}_{\leq j_{r_{i}}}, x^{T}_{i}, y^{T}_{j_1}, \ldots, y^{T}_{j_{r_{i}}} \right)$   \;
}

\end{AlgoEq}

  \caption{Direct method to solve system involving $\mathcal{B}_{\leq i}$}
  \label{alg:arrowhead_direct_method}
\end{algorithm2e}

\begin{algorithm2e}[ht!]
  \LinesNumbered
\SetCommentSty{itshape}
\DontPrintSemicolon

\SetKwFunction{computeArrowheadSchur}{computeArrowheadSchur}
\SetKwFunction{outgoingSchurBlock}{outgoingSchurBlock}

\begin{AlgoEq}
  \Fn{computeArrowheadSchur($i$)}{

    \Input{
      Vertex $i \in \Vertices$
    }

    \Output{
      Schur complements $(\mathcal{S}_{\leq j})_{j \in \Vertices_{\leq i}}$
    }

    \If{Vertex $i$ is a leaf in $\Graph$}{
      \Return $\emptyset$
    }

    Compute subtree Schur complements:
    \begin{equation*}
      \begin{aligned}
        (\mathcal{S}_{\leq k})_{k \in \Vertices_{\leq j}} \gets {}& \computeArrowheadSchur(j) \\
                                                               & \quad \forall \: j \in V_i = \{j \in \Vertices_{\leq i} \mid j \neq i \}
      \end{aligned}
    \end{equation*}

    $\left( \mathcal{S}_{\leq i} \right)_{l, l} \gets 0 \quad \forall \: l, l' \in \{1, \ldots, r_{i}\} $ \;

    \ForAll{$l \in \{1, \ldots, r_{i}\}$}{
      $Z_l \gets B_{i}^{-1} \left( \outMat{j_{l}} \right)^{T} $ \;
      \label{alg:arrowhead_schur:outgoing_solve}
    }

    \ForAll{$l \in \{1, \ldots, r_{i}\}$}{

      \ForAll{$l' \in \{1, \ldots, r_{i}\}$}{
        $\left( \mathcal{S}_{\leq i} \right)_{l, l'} \gets \left( \mathcal{S}_{\leq i} \right)_{l, l'} + \outMat{j_{l}} Z_{l'}$
      }

      $S_l \gets \outgoingSchurBlock(j_{l}, \mathcal{B}_{\leq j_l}, \mathcal{S}_{\leq j_l})$ \;
      $\left( \mathcal{S}_{\leq i} \right)_{l, l} \gets \left( \mathcal{S}_{\leq i} \right)_{l, l} +  S_l + D_{j_l}$ \;

    }

    \Return $(\mathcal{S}_{\leq j})_{j \in \Vertices_{\leq i}}$ \;
  }
\end{AlgoEq}

  \caption{Algorithm to compute Schur complement $\Schur_{\leq i}$}
  \label{alg:arrowhead_schur}
\end{algorithm2e}
\begin{algorithm2e}[ht!]
  \LinesNumbered
\SetCommentSty{itshape}
\DontPrintSemicolon

\begin{AlgoEq}

  \Fn{outgoingSchurBlock($i$, $\mathcal{B}_{\leq i}$, $\mathcal{S}_{\leq i}$)}{
    \Input{
      \parbox[t]{\hsize}{
        Vertex $i \in \Vertices, i \neq r_{i}$ \\
        System $\mathcal{B}_{\leq i}$ \\
        Schur complement $\mathcal{S}_{\leq i}$
      }
    }

    \Output{
      \parbox[t]{\hsize}{
        Schur block $\mathcal{C}_{i}^{-} \mathcal{B}_{i}^{-1} (\mathcal{C}_{i}^{-})^{T}$
        }
    }

    $(h_1, \ldots, h_{\numCoupled_{k_i}}) \gets \text{Columns of matrix } (\incMat{k_i})^{T}$ \;

    \ForAll{$l \in \{1, \ldots, \numCoupled_{k_i}\}$}{

      Solve system to obtain $z_{l} \gets B_i^{-1} h_l$ \;

      Compute right-hand sides for each $l' \in \{1, \ldots, r_{i}\}$:
      \label{alg:arrowhead_schur:first_incoming_solve}
      \begin{equation*}
        \hat{x}_{l'}^{l} \gets \outMat{k_{l'}} z_{l}
      \end{equation*}

      Solve Schur complement $(y_{j_1}^{l}, \ldots, y_{j_{r_{i}}}^{l}) \gets \mathcal{S}_{\leq i}^{-1}(\hat{x}^{l}_1, \ldots, \hat{x}^{l}_{r_{i}})$ \;

      Compute solution $x_{i}^{l} \gets B_{i}^{-1} (h_{l} - \sum_{l = 1}^{r_{i}} (\outMat{k_{l}})^{T} y_{j_{l}})$
      \label{alg:arrowhead_schur:second_incoming_solve}
      \;
    }

    $Z_{i} \gets \text{Matrix consisting of columns } x_{i}^{1}, \ldots, x_{i}^{\numCoupled_{k_i}}$ \;

    \Return $\incMat{k_i} Z_{i}$
  }

\end{AlgoEq}

  \caption{Algorithm to compute outgoing Schur complement $\mathcal{C}_{i}^{-} \mathcal{B}_{i}^{-1} (\mathcal{C}_{i}^{-})^{T}$}
  \label{alg:arrowhead_schur_outgoing}
\end{algorithm2e}
Based on Assumption~\ref{ass:direct_method_assumptions}, we propose
the direct method laid out in
Algorithm~\ref{alg:arrowhead_direct_method}, solving
a system with matrix $\mathcal{B}_{\leq i}$ using the Schur complement $\Schur_{\leq i}$ and
recursing into the subtree rooted at $i$.
The algorithm consists of three steps:
First, the right-hand sides $\hat{x}_{l}$ for the Schur complement are
computed recursively. Then the system $\Schur_{\leq i} y = \hat{x}$ is solved,
yielding the portion of the solution associated with
the coupling variables $y_{j_l}$. Finally, the remaining parts of
the solution, $x_{i}$ and $x_{\leq j_l}$, are recursively computed based on the
variables $y_{j_l}$.  The first and third steps involve a recursion
into the subtrees rooted at the head vertices of the arcs in
$\outArcs{i}$.  To solve~\eqref{eq:2by2_coupled_system} in its
entirety, we employ the algorithm at the root $\rootVert$ of $\Graph$.

To apply Algorithm~\ref{alg:arrowhead_direct_method},
we need a method to compute the Schur complements $\Schur_{\leq i}$.
Simple calculations show that $\Schur_{\leq i}$ is given
by blocks
\begin{equation*}
  \left( \Schur_{\leq i} \right)_{l, l'} =
  \begin{cases}
    \outMat{k_l} B_{i}^{-1} \left( \outMat{k_l} \right)^{T}
    + \mathcal{C}^{-}_{j_l} \mathcal{B}^{-1}_{\leq j_l} \left( \mathcal{C}^{-}_{j_l} \right)^{T}
    + D_{j_l} & \text{ if } l = l' \text{, and} \\
    \outMat{k_l} B_{i}^{-1} ( \outMat{k_{l'}} )^{T} & \text{ otherwise},
  \end{cases}
\end{equation*}
for $l, l' \in \{1, \ldots, r_{i}\}$. The occurrence of the inverses of the
matrices $\mathcal{B}^{-1}_{\leq j_l}$ suggests that the Schur complements can be
computed from the leaves of the tree up towards its root.

At a leaf $i \in \Vertices$,
no Schur complement is required and~Algorithm~\ref{alg:arrowhead_direct_method}
only solves a system with matrix $B_i$.
To explain the computation higher up in the tree, we put ourselves in
the position where at vertex $i \in \Vertices$, $\Schur_{\leq i}$
has already been computed. At this stage, all that remains
to be done at vertex $i$ is the computation of the product
$\mathcal{C}^{-}_{i} \mathcal{B}^{-1}_{\leq i} ( \mathcal{C}^{-}_{i}
)^{T}$, required for some diagonal block of the Schur complement of the
parent vertex of $i$ (assuming of course that $i \neq \rootVert$, in which
case all Schur complements are by now computed).
To compute this product, we could simply use Algorithm~\ref{alg:arrowhead_direct_method}
repeatedly to solve systems with the matrix $\mathcal{B}_{\leq i}$
and right-hand sides defined using the rows of $\mathcal{C}^{-}_{i}$. However, the structure of the product
makes the computations significantly easier:

First, the only non-zero block $\mathcal{C}^{-}_{i}$ is at the
position of $B_i$ with respect to $\mathcal{B}_{\leq i}$.  Thus, the
right-hand sides $h_{\leq i}$ for the computation have the property
that $h_{\leq j_l} = 0$ and $f_{j_l} = 0$ for all
$l \in \{1, \ldots, r_{i}\}$.  Second, since we only require the product
of the solution with $\mathcal{C}^{-}_{i}$, we only have to compute
the solution value of $x_i$ and can skip the computation of
$x_{\leq j_l}$ and $y_{j_l}$ for all $l \in \{1, \ldots, r_{i}\}$.  As a
consequence, no recursion is required at all in order to determine the
required product, greatly accelerating the computations. The
specific computations to be carried out are summarized
in Algorithms~\ref{alg:arrowhead_schur} and \ref{alg:arrowhead_schur_outgoing}.

\subsection{Complexity}

In the following, we examine the running time of both
Algorithms~\ref{alg:arrowhead_direct_method}
and~\ref{alg:arrowhead_schur} used to solve
\eqref{eq:tree_coupled_system} and determine the required Schur
complements $\Schur_{\leq i}$, respectively.

In general, the complexity of the algorithms depends on how the
systems with matrices $B_i$ and $\Schur_{\leq i}$ are solved, which we have not
specified so far. Since we make no assumptions regarding the matrices
$B_i$ other than non-singularity, any appropriate method may be used in
practice. On the other hand, since the Schur complements
$\Schur_{\leq i}$ are assumed to be positive definite, Cholesky
decompositions or the Conjugate Gradient method \cite{cg} are likely to be used.

To perform an analysis of the complexity independently of these
details we make two assumptions: First, since $\numCoupled_k$ is small
compared to the sizes $\blockSize_i$ of the matrices $B_{i}$, we assume
that the computational costs are dominated by the solutions
of systems $B_i x_{i} = y_{i}$ with different right-hand sides $y_i$, ignoring
other parts such as matrix--vector products with $\incMat{k}$ and $\outMat{k}$
and the solution of systems involving $\Schur_{\leq i}$. Second,
since the methods used to solve the systems $B_{i}$ are arbitrary,
we simply assume that the solution of a single
system involving $B_{i}$ can be carried out in constant time and count
the total number of  such solves, finding
the following:
\begin{lem}
  \label{lem:direct_running_time}
  For each vertex $i \in \Vertices$ the following holds:
  \begin{enumerate}
  \item
    The solution of system~\eqref{eq:tree_coupled_system} using
    Algorithm~\ref{alg:arrowhead_direct_method} with precomputed
    Schur complements requires
    $\Theta( 2^{\depth(i)})$ solutions of systems involving $B_i$.
  \item
    The computation of the Schur complements using
    Algorithm~\ref{alg:arrowhead_schur} requires
    $\mathcal{O}(\numCoupled_{i})$ solutions of systems involving $B_i$.
  \end{enumerate}
\end{lem}

\begin{proof}
  \begin{enumerate}
  \item
    Let us examine a vertex $i \in \Vertices$ at a depth of $l \geq 0$. Let
    $\rootVert = i_{1}, i_2, \ldots, i_{l + 1} = i$ be the vertices on the
    $(\rootVert, i)$-path in $\Graph$. During the execution of
    Algorithm~\ref{alg:arrowhead_direct_method} at the parent vertex
    $i_{l}$, the algorithm recurses into vertex $i_{l + 1}$ exactly twice:
    First during the computation of the right-hand sides $\hat{x}_{l}$ in
    line~\ref{alg:arrowhead_direct_method:first_recursion}, and
    then in line~\ref{alg:arrowhead_direct_method:second_recursion}
    to assemble the solution $x_{\leq i_{l}}$ based on the
    solution of the Schur complement. Thus, each time
    Algorithm~\ref{alg:arrowhead_direct_method} is executed at vertex $i_{l}$,
    it recurses into vertex $i$ twice. The same holds for $i_{l}$ and its
    parent vertex $i_{l - 1}$ and so on up to the root $\rootVert$ where
    the algorithm is executed once. In total,
    Algorithm~\ref{alg:arrowhead_direct_method} is executed $2^{l}$ times for vertex $i$.
    Each time the algorithm is executed at vertex $i$, a system involving $B_i$ is solved once
    at line~\ref{alg:arrowhead_direct_method:leaf_solve} if $i$ is a leaf
    and twice at lines~\ref{alg:arrowhead_direct_method:inner_first_solve}
    and~\ref{alg:arrowhead_direct_method:inner_second_solve} respectively
    if $i$ is an inner vertex, yielding a total of $\Theta(2^{l})$ solves.
  \item The result follows from the structure of
    Algorithm~\ref{alg:arrowhead_schur}: The computation
    of the Schur complement $\Schur_{\leq i}$ does not require a
    recursion beyond the child vertices of $i$ itself. Consequently,
    both functions in the algorithm are executed exactly once for
    each vertex $i$. At vertex $i$ a system involving $B_{i}$ is solved once
    for each of the rows of $\outMat{k_{l}}$ for all $l \in \{ 1, \ldots, r_{i} \}$ at
    line~\ref{alg:arrowhead_schur:outgoing_solve} and twice
    for each row in $\incMat{k_{i}}$ at lines
    \ref{alg:arrowhead_schur:first_incoming_solve} and
    \ref{alg:arrowhead_schur:second_incoming_solve} of
    Algorithm~\ref{alg:arrowhead_schur_outgoing}
    respectively,
    yielding a total number of $\mathcal{O}(\numCoupled_{i})$ solves.
    \qedhere
  \end{enumerate}
\end{proof}
Based on these bounds it is clear that the total number
solved systems involving $B_{i}$ summed over all vertices
during an application of
Algorithm~\ref{alg:arrowhead_direct_method} at $\rootVert$ is in
$\Theta(2^{\heightGraph})$ where systems at the leaves of $\Graph$ are
solved most often. This running time is exponential rather than
polynomial for general graphs $\Graph$, for which
$\heightGraph \in \Theta(\numBlocks)$.  Specifically, a worst-case
instance for the algorithm consists of a tree composed of a single
path having a height of $\numBlocks - 1$. Conversely, if the height is
sufficiently small compared to the input size of the problem, the
complexity remains polynomially bounded. This holds in
particular for full proper trees, where height grows
logarithmically in the number of vertices.
In contrast to the solution of the system based on
Algorithm~\ref{alg:arrowhead_direct_method}, the computation of the
Schur blocks using Algorithm~\ref{alg:arrowhead_schur} is polynomial
regardless of the height of $\Graph$, which is due to the fact that
the recursion is much more limited during the computation of the Schur
complement blocks.

Lastly, we would like to put these results in the context of those
obtained in~\cite{parallel_interior}: Firstly,~\citeauthor{parallel_interior}
consider the `symmetric bordered
block-diagonal structure' as one of several exploitable structures,
another one being `rank-$k$ correcting matrices'. They examine
these structures from a software design framework and show how they
can be seen as a set of data structures with common functionality
which can be mapped onto a class hierarchy in object-oriented
software. Consequently, since these data structures are expected to behave
like interchangeable black-box components, the corresponding analysis
of their interplay is limited. For instance, the complexity
laid out here is not examined in~\cite{parallel_interior} and
the computation of the required Schur complements according
to their framework would also incur an exponential running time,
as opposed to the polynomial running time of Algorithm~\ref{alg:arrowhead_schur}.


\section{Nested Arrowhead Structure}
\label{sec:recursive_preconds}

\subsection{A Block-Diagonal Preconditioner}
\label{sec:nested_block_diagonal}

In this section, we continue the investigation of the nested arrowhead
structure introduced in Section~\ref{sec:direct_method}. While this
section was based on a thorough examination of the method
by \citeauthor{parallel_interior}, we now proceed to
introduce several entirely new approaches based on problem-specific preconditioners,
which we define recursively from the leaves of the graph $\Graph$ towards its root.
Specifically, we let
$\Precond_{\leq i}$ be the preconditioner associated with the
subtree of $\Graph$ rooted at $i \in \Vertices$, where we
denote by $\Precond$ the preconditioner $\Precond_{\leq \rootVert}$ for the entire tree.
If a vertex $i \in \Vertices$ is a leaf, we always
let $\Precond_{\leq i} \define B_{i}$. If $i$ is a inner vertex we let
$\outArcs{i} = (a_{k_{1}} = (i, j_{1}), \ldots, a_{k_{r_{i}}} = (i, j_{r_{i}}))$ be
the outgoing arcs of $i$ and define $\Precond_{\leq i}$ recursively
based on the preconditioners $\Precond_{\leq j_{l}}$. Our first preconditioner is
called the \emph{block-diagonal preconditioner} defined by
\begin{equation}
  \label{eq:block_diagonal_preconditioner}
  \Precond^{\diag}_{\leq i} \define \blkdiag \left( \Precond^{\diag}_{\leq j_{1}}, \ldots, \Precond^{\diag}_{\leq j_{r_{i}}}, B_{i}, - \mathcal{D}_{\leq i} \right).
\end{equation}
This preconditioner is fairly simple in its structure, approximating
the matrices $B_i$ while disregarding the coupling matrices
$\mathcal{C}_{\leq i}^{-}$ and $\mathcal{C}_{\leq i}^{+}$
completely. This preconditioner can, however, only be applied if all
matrices $D_{k}$ are invertible, which holds upon application of the sequential
homotopy or interior point methods but not for KKT systems arising from quadratic programs
of an SQP method.

\subsection{The Hook Preconditioner}
\label{sec:nested_hook}
A class of preconditioners can be derived based on the Schur complements
$\Schur_{\leq i}$ introduced above. Recall that as opposed to the
direct method in Algorithm~\ref{alg:arrowhead_direct_method}, these
preconditioners can be efficiently computed using
Algorithm~\ref{alg:arrowhead_schur}. We, of course, again work
under Assumption~\ref{ass:direct_method_assumptions} to ensure the
Schur complements exist and are positive definite.
The \emph{hook preconditioner} $\Precond^{\hook}$ is defined as
\begin{equation}
  \label{eq:hook_preconditioner}
  \Precond^{\hook}_{\leq i}
  \define
  \adjustbox{valign=c}{
    \begin{tikzpicture}[inner sep=0pt, minimum width=0cm]
      \begin{scope}
        \matrix (R) [every node/.style={minimum width=.9cm, minimum height=.6cm}, matrix of math nodes, left delimiter=(, right delimiter=)]
      {
        |(block-begin)| \Precond^{\hook}_{\leq  j_{1}} & & & & \\
        & \ddots & & & \\
        & & |(pcenter)| \Precond^{\hook}_{\leq  j_{r_{i}}} & & \\
        & & & |(block-end)| B_i & \\
        |(C-begin)| & -\mathcal{C}_{\leq i}^{-} & |(C-end)| & \mathcal{C}_{\leq i}^{+} & -\Schur_{\leq i} \\
      };
    \end{scope}
    \begin{scope}[layer=background layer]
      \fill[black!5,rounded corners=0.1cm] (C-begin.north west) rectangle (C-end.south east) {};
    \end{scope}
    \end{tikzpicture}
  }.
\end{equation}
Note that while $\Precond^{\hook}_{\leq i}$ is invertible under
Assumption~\ref{ass:direct_method_assumptions}, it is not symmetric,
necessitating the use of GMRES~\cite{gmres}, for instance, rather than the
MINRES method~\cite{minres} which is less operation-intensive per iteration.
For any inner vertex $i$ the application of the inverse of
$\Precond^{\hook}_{\leq i}$ to a given right-hand side
$h^{T}_{\leq i} = (h^{T}_{\leq j_{1}}, \ldots, h^{T}_{\leq j_{r_{i}}},\allowbreak h^{T}_{i},\allowbreak
f^{T}_{j_{1}}, \ldots, f^{T}_{j_{r_{i}}})$ is carried out by conducting the
following steps:
\begin{enumerate}
\item
  Recursively compute products
  $x_{\leq j_{l}} \gets (\Precond^{\hook}_{\leq j_l})^{-1} h_{\leq j_{l}}$
  for $l = 1, \ldots, r_{i}$ and $x_{i} \gets B_{i}^{-1} h_{i}$.
\item
  Compute right-hand side
  \begin{equation*}
    \hat{z} \gets
    - \mathcal{C}_{\leq i}^{-}
    \begin{pmatrix}
      x_{\leq j_{1}}^{T} & \ldots & x_{\leq j_{r_{i}}}^{T}
    \end{pmatrix}^{T}
    + \mathcal{C}_{\leq i}^{+} x_{i}
    -
    \begin{pmatrix}
      f_{ j_{1}}^{T} & \ldots & f_{ j_{r_{i}}}^{T}
    \end{pmatrix}^{T}.
  \end{equation*}
\item
  Compute solution
  \begin{equation*}
    \begin{pmatrix}
      y_{i_{1}}^{T} & \ldots & y_{i_{r_{i}}}^{T}
    \end{pmatrix}^{T}
    \gets \Schur_{\leq i}^{-1} \hat{z}.
  \end{equation*}
\end{enumerate}
Let us now turn to the theoretical properties of $\Precond^{\hook}_{\leq  i}$.
To begin, we observe that if $i$ is a leaf of $\Graph$,
then $\Precond^{\hook}_{\leq  i} = B_{i}$ is a perfect preconditioner, which yields the correct solution after
a single preconditioned step. If $i$ has a height of one rather than being a leaf, we see
that
\begin{equation*}
  \Precond^{\hook}_{\leq  i}
  =
  \adjustbox{valign=c}{
    \begin{tikzpicture}[inner sep=0pt]
      \matrix (R) [every node/.style={minimum width=.9cm, minimum height=.6cm}, matrix of math nodes,left delimiter=(, right delimiter=)]
      {
        |(block-begin)| B_{j_{1}} (= \mathcal{B}_{\leq j_{1}}) & & & & \\
        & \ddots & & & \\
        & & B_{j_{r_{i}}} (= \mathcal{B}_{\leq j_{r_{i}}}) & & \\
        & & & |(block-end)| B_i & \\
        |(C-begin)| & -\mathcal{C}_{\leq i}^{-} & |(C-end)| & \mathcal{C}_{\leq i}^{+} & - \Schur_{\leq i} \\
      };
      \begin{scope}[layer=background layer]
        \fill[black!5,rounded corners=0.1cm] (C-begin.north west) rectangle (C-end.south east) {};
      \end{scope}
    \end{tikzpicture}
  }.
\end{equation*}
Simple calculations show that the preconditioned matrix is then given by
\begin{equation}
  \label{eq:hook_preconditioned_height_one}
  (\Precond^{\hook}_{\leq  i})^{-1} \mathcal{B}_{\leq i}
  =
  \adjustbox{valign=c}{
    \begin{tikzpicture}[inner sep=0pt]
      \matrix (R) [every node/.style={minimum width=.9cm, minimum height=.6cm}, matrix of math nodes,left delimiter=(, right delimiter=)]
      {
        |(block-begin)| \identMat & & & & |(Z-begin)| \\
        & \ddots & & & -\mathcal{Z}^{-}_{\leq i} \\
        & & \identMat & & |(Z-end)| \\
        & & & |(block-end)| \identMat & \mathcal{Z}^{+}_{\leq i} \\
        & & & & \identMat \\
      };
      \begin{scope}[layer=background layer]
        \fill[black!5,rounded corners=0.1cm] (Z-begin.north west) rectangle (Z-end.south east) {};
      \end{scope}
    \end{tikzpicture}
  },
\end{equation}
where $\mathcal{Z}^{-}_{\leq i} \define \blkdiag(B_{j_{1}}^{-1} (C_{k_{1}})^{T}, \ldots, B_{j_{r_{i}}}^{-1} (C_{k_{r_{i}}})^{T})$
and $\mathcal{Z}^{+}_{\leq i} \define B_{i}^{-1} (\mathcal{C}_{\leq i}^{+})^{T}$.
Consequently, we can make the following observation:
\begin{lem}
  \label{lem:hook_depth_one}
  If the height of vertex $i$ is one, then the minimal polynomial of the
  preconditioned matrix~\eqref{eq:hook_preconditioned_height_one}
  is given by%
  \begin{equation*}
    \mu_{(\Precond^{\hook}_{\leq i})^{-1} \mathcal{B}_{\leq i}}(x) = (1 - x)^2.
  \end{equation*}
\end{lem}
\begin{proof}
  The minimal polynomial of the
  preconditioned matrix~\eqref{eq:hook_preconditioned_height_one} is given by $(1 - x)^2$ since the square of
  the matrix
  \begin{equation*}
    \identMat -  (\Precond^{\hook}_{\leq  i})^{-1} \mathcal{B}_{\leq i} =
    \adjustbox{valign=c}{
    \begin{tikzpicture}[inner sep=0pt]
      \matrix (R) [every node/.style={minimum width=.9cm, minimum height=.6cm}, matrix of math nodes,left delimiter=(, right delimiter=)]
      {
        |(block-begin)|\zeroMat & & & & |(Z-begin)| \\
        & \ddots & & & \mathcal{Z}^{-}_{\leq i} \\
        & & \zeroMat & & |(Z-end)| \\
        & & & |(block-end)| \zeroMat & - \mathcal{Z}^{+}_{\leq i} \\
        & & & & \zeroMat \\
      };
      \begin{scope}[layer=background layer]
        \fill[black!5,rounded corners=0.1cm] (Z-begin.north west) rectangle (Z-end.south east) {};
      \end{scope}
    \end{tikzpicture}
  }
\end{equation*}
  evaluates to zero.
\end{proof}
\begin{cor}
  \label{cor:hook_convergence}
  If the height of vertex $i$ is one, then the  GMRES method applied
  to $\mathcal{B}_{\leq i}$ with preconditioner $\Precond^{\hook}_{\leq  i}$
  converges in at most two iterations in exact arithmetic.
\end{cor}
\begin{proof}
  See~\cite[Proposition 2]{gmres}.
\end{proof}
\subsection{Recursive Preconditioning}
\label{sec:nested_recursive}
Naturally, for vertices with larger heights, the performance of the
hook preconditioner can be assumed to degrade. However, we can make use
of the preconditioner within an iterative scheme in order to solve the
system $\mathcal{B}_{\leq i} x_{\leq i} = h_{\leq i}$. To this end,
we define $\Precond^{\recursive}_{\leq  i}$ to be a \emph{recursive preconditioner}
if it satisfies the following conditions:
\begin{enumerate}
\item
  If $i$ is a leaf of $\Graph$, then it holds that $\Precond^{\recursive}_{\leq  i} = B_{i}$.
\item
  If $i$ is an inner vertex, then $\Precond^{\recursive}_{\leq  i}$  satisfies the following recursion:
  \begin{equation}
    \label{eq:recursive_preconditioner}
    \Precond^{\recursive}_{\leq  i}
    \define
    \adjustbox{valign=c}{
      \begin{tikzpicture}[inner sep=0pt]
        \matrix (R) [every node/.style={minimum width=.8cm, minimum height=.6cm}, matrix of math nodes,left delimiter=(, right delimiter=)]
        {
          |(block-begin)| \Precond^{\recursive}_{\leq  j_{1}} & & & & \\
          & \ddots & & & \\
          & & \Precond^{\recursive}_{\leq  j_{r_{i}}} & & \\
          & & & |(block-end)| B_i & \\
          |(C-begin)| & -\mathcal{C}_{\leq i}^{-} & |(C-end)| & \mathcal{C}_{\leq i}^{+} & -\Schur_{\leq i} \\
        };
        \begin{scope}[layer=background layer]
          \fill[black!5,rounded corners=0.1cm] (C-begin.north west) rectangle (C-end.south east) {};
        \end{scope}
      \end{tikzpicture}
    }.
  \end{equation}
\end{enumerate}
Based on a recursive preconditioner $\Precond^{\recursive}_{\leq i}$ and an initial solution $x_{\leq i}^{(0)}$, we consider
a linear fixed-point iteration based on the approximation
$(\Precond^{\recursive}_{\leq  i})^{-1} \approx \mathcal{B}_{\leq i}^{-1}$, which is given by
\begin{equation}
  \label{eq:recursive_iterative_method}
  x_{\leq i}^{(k + 1)} \gets x_{\leq i}^{(k)} + (\Precond_{\leq  i}^{\recursive})^{-1} \left( h_{\leq i} - \mathcal{B}_{\leq i} x_{\leq i}^{(k)}  \right).
\end{equation}
We can make the following observation:
\begin{lem}
  \label{lem:exact_recursive_solution}
  Let $i$ be a vertex of $\Graph$. If the recursive preconditioner
  $\Precond^{\recursive}_{\leq i}$ is exact on the child vertices $j_{1}, \ldots, j_{r_{i}}$ of $i$,
  \ie
  \begin{equation*}
    \Precond^{\recursive}_{\leq  j_{l}} = \mathcal{B}_{\leq j_{l}}
    \quad  \forall \:l = 1, \ldots, r_{i},
  \end{equation*}
  then the iterative method given by
  \eqref{eq:recursive_iterative_method} converges in at most two
  iterations in exact arithmetic, independently of the initial iterate $x_{\leq i}^{(0)}$.
\end{lem}
\begin{proof}
  If $i$ is a leaf, then $\Precond^{\recursive}_{\leq i} = B_{i}$ must be
  exact. Otherwise, simple calculations show that the preconditioned matrix is given by
  \begin{equation*}
    (\Precond^{\recursive}_{\leq  i})^{-1} \mathcal{B}_{\leq i}
    =
    \adjustbox{valign=c}{
      \begin{tikzpicture}[inner sep=0pt]
        \matrix (R) [every node/.style={minimum width=.9cm, minimum height=.6cm}, matrix of math nodes,left delimiter=(, right delimiter=)]
        {
          |(block-begin)| \identMat & & & & |(Y-begin)| \\
          & \ddots & & & -\mathcal{Y}^{-}_{\leq i} \\
          & & \identMat & & |(Y-end)| \\
          & & & |(block-end)| \identMat & \mathcal{Z}^{+}_{\leq i} \\
          & & & & \identMat \\
        };
        \begin{scope}[layer=background layer]
          \fill[black!5,rounded corners=0.1cm] (Y-begin.north west) rectangle (Y-end.south east) {};
        \end{scope}
      \end{tikzpicture}
    },
  \end{equation*}
  where $\mathcal{Y}^{-}_{\leq i} \define \blkdiag((\Precond^{\recursive}_{\leq j_{1}})^{-1} (C_{k_{1}})^{T}, \ldots, (\Precond^{\recursive}_{\leq j_{r_{i}}})^{-1} (C_{k_{r_{i}}})^{T})$
  and $\mathcal{Z}^{+}_{\leq i}$ is defined as above. As was the case before,
  the minimal polynomial of the preconditioned matrix is
  given by
  \begin{equation*}
    \mu_{(\Precond_{\leq  i}^{\recursive})^{-1} \mathcal{B}_{\leq i}}(x) = (1 - x)^2,
  \end{equation*}
  which of course implies that $(\identMat - (\Precond_{\leq  i}^{\recursive})^{-1} \mathcal{B}_{\leq i})^2 = \zeroMat$.
  Consequently, the first two computed iterates are given by
  \begin{equation*}
    \begin{aligned}
      x_{\leq i}^{(1)} &= (\identMat - (\Precond^{\recursive}_{\leq i})^{-1} \mathcal{B}_{\leq i}) x_{\leq i}^{(0)} + (\Precond^{\recursive}_{\leq i})^{-1} h_{\leq i}, \\
      x_{\leq i}^{(2)} &= (\identMat - (\Precond^{\recursive}_{\leq i})^{-1} \mathcal{B}_{\leq i})^{2} x_{\leq i}^{(0)} + (2 \identMat - (\Precond^{\recursive}_{\leq i})^{-1} \mathcal{B}_{\leq i})(\Precond^{\recursive}_{\leq i})^{-1} h_{\leq i} \\
      &= (2 \identMat - (\Precond^{\recursive}_{\leq i})^{-1} \mathcal{B}_{\leq i})(\Precond^{\recursive}_{\leq i})^{-1} h_{\leq i}.
    \end{aligned}
  \end{equation*}
  Note that the value of $x_{\leq i}^{(2)}$ is independent of the initial
  solution $x_{\leq i}^{(0)}$. This holds in particular
  if we set
  $x_{\leq i}^{(0)} = \mathcal{B}_{\leq i}^{-1} h_{\leq i} \enifed x^{*}$,
  in which case it follows that $x_{\leq i}^{(k)} \equiv x^{*}$ for all $k \in \Nat$
  and specifically for $k = 2$. However, due to the independence of $x_{\leq i}^{(2)}$
  it must hold that $x_{\leq i}^{(2)} = x^{*}$ for \emph{all} choices of $x_{\leq i}^{(0)}$.
\end{proof}

\subsection{Exact Preconditioning}
\label{sec:nested_exact}
Based on the results obtained previously, we can obtain an exact polynomial
preconditioner $\Precond^{\exact}$ as an alternative to the direct
method introduced in Section~\ref{sec:direct_method} based on
Lemma~\ref{lem:exact_recursive_solution}. We can build such a
preconditioner bottom up, starting with single vertices while
maintaining the conditions of Lemma~\ref{lem:hook_depth_one}. For each leaf
$i \in \Vertices$, we set $\Precond^{\exact}_{\leq i} \define B_i$. For each
vertex of height one, we consider the recursive
preconditioner $\Precond^{\recursive}_{\leq i}$ following the construction~\eqref{eq:recursive_preconditioner},
using $\Precond^{\exact}_{\leq j_l}$ for $l = 1, \ldots, r_{i}$.
This preconditioner satisfies the conditions of
Lemma~\ref{lem:exact_recursive_solution} yielding
the exact solution of $\mathcal{B}_{\leq i} x_{\leq i} = h_{\leq i}$
by applying~\eqref{eq:recursive_iterative_method} twice. This recursion is linear
in the right-hand side $h_{\leq i}$ and can be used to define $\Precond^{\exact}_{\leq i}$:
\begin{equation*}
  x_{\leq i}^{(2)} = (2 \identMat - (\Precond^{\recursive}_{\leq i})^{-1} \mathcal{B}_{\leq i}) (\Precond^{\recursive}_{\leq i})^{-1} h_{\leq i} \enifed (\Precond^{\exact}_{\leq i})^{-1} h_{\leq i}.
\end{equation*}
For vertices of height two, we again perform two iterations using the
recursive construction~\eqref{eq:recursive_preconditioner} based on
the newly defined exact preconditioners for vertices of height one,
yielding an exact preconditioner $\Precond^{\exact}_{\leq i}$ for each vertex $i \in \Vertices$ of height two. We
follow this process until we reach the root $\rootVert$,
obtaining an exact preconditioner for the recursive variant of the
original problem~\eqref{eq:2by2_coupled_system}:
\begin{cor}
  \label{cor:exact_convergence}
  $\Precond^{\exact}_{\leq i}$ is exact for every vertex $i \in \Vertices$.
\end{cor}

\subsection{Complexity}

In the following, we consider the complexity of applying
$(\Precond^{\hook})^{-1}$ and $(\Precond^{\exact})^{-1}$ in terms of the number of solutions
of systems involving $B_{i}$ analogous to Section~\ref{sec:direct_method}. We exclude
the computation of the Schur complements $\Schur_{\leq i}$ because
their complexity has been established in Lemma~\ref{lem:direct_running_time}
and find the following results:
\begin{lem}
  \begin{enumerate}
  \item
    The application of the inverse of $\Precond^{\hook}$ requires
    exactly one solution of a system involving $B_{i}$ for each vertex $i \in \Vertices$.
  \item
    The application of the inverse of $\Precond^{\exact}$ requires
    $\Theta( 2^{\depth(i)})$ solutions of systems involving $B_i$ for each
    vertex $i \in \Vertices$.
  \end{enumerate}
\end{lem}
\begin{proof}
  \begin{enumerate}
  \item
    To apply $(\Precond^{\hook}_{\leq i})^{-1}$ at a vertex $i \in \Vertices$, we need
    to solve exactly one system involving $B_{i}$ and recurse into each child
    vertex of $i$ exactly once. Therefore, to apply $\Precond^{\hook}$ at the
    root vertex $\rootVert$, we have to solve exactly one system involving $B_{i}$.
  \item
    Our reasoning is analogous to the proof of Lemma~\ref{lem:direct_running_time}:
    To apply $(\Precond^{\exact}_{\leq i})^{-1}$ at a leaf $i \in \Vertices$, we
    solve a system involving $B_{i}$ exactly once. At an inner vertex $i$, we need to
    compute the product of
    \begin{equation*}
      (\Precond^{\exact}_{\leq i})^{-1} = (2 \identMat - (\Precond^{\recursive}_{\leq i})^{-1} \mathcal{B}_{\leq i}) (\Precond^{\recursive}_{\leq i})^{-1}
    \end{equation*}
    with some right-hand side. The computation of the product
    involves two solves systems involving $B_{i}$ itself as well
    as two recursions into each of the subtrees rooted at
    the child vertices $j_{1}, \ldots, j_{r_{i}}$ of $i$ in
    order to apply the inverses of the respective preconditioners $\Precond^{\exact}_{\leq j_{l}}$.
    Consequently, when moving up the tree from $i$ towards $\rootVert$,
    the number of solves of systems involving $B_{i}$ doubles each time we move from
    a vertex to its parent leading to the stated complexity.  \qedhere
  \end{enumerate}
\end{proof}
It is apparent from these results that the application of the inverse
of $\Precond^{\exact}$ has the same asymptotic complexity as the
direct method introduced in Section~\ref{sec:direct_method}, which is
exponential in the height of $\Graph$ and therefore only polynomial
for trees with logarithmic heights. Conversely, the inverse of
$\Precond^{\hook}$ can always applied in polynomial time.

\section{Non-nested Arrowhead Structure}
\label{sec:direct_preconds}

In the following, we examine the original
problem~\eqref{eq:2by2_coupled_system} more closely, particularly
focusing on the Schur complement~\eqref{eq:2by2_schur_complement},
which we assume to be positive definite. We derive a
preconditioner $\Precond$ based on the non-nested Schur complement and focus on
the efficient application of an approximate inverse of this preconditioner.
We first note that the Schur complement $\Schur$ consists of blocks of
sizes $\numCoupled_{k} \times \numCoupled_{k'}$ for each pair $(a_{k}, a_{k'})$ of arcs,
where a block is non-zero iff the arcs share a vertex.
It is however advantageous to instead examine $\Schur$
using vertex-based blocks, each of which combines the variables
of the outgoing arcs of the respective vertices. Consequently,
the only non-trivial vertex-based blocks appear for inner vertices
of the inner subgraph $\innerGraph$.

\begin{figure}[ht]
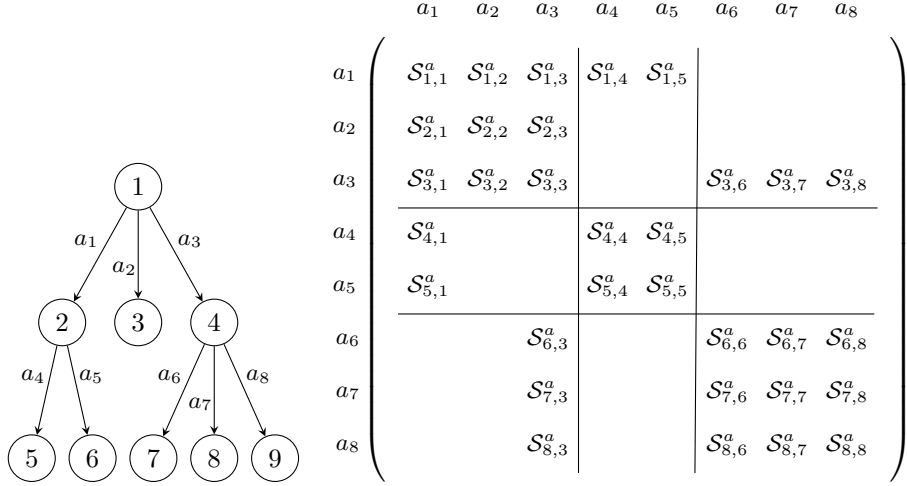

  \centering
  \begin{subfigure}{.35\textwidth}
    \centering
    \includegraphics{nonnested_tree}
  \end{subfigure}
  \begin{subfigure}{.64\textwidth}
    \centering
    \includegraphics{nonnested_tree_system}
  \end{subfigure}
  \caption{
    A graph $\Graph$ together with the Schur complement $\Schur$. The eight
    arc-based blocks (denoted here by $\Schur^{a}_{k, k'}$) can be combined into three vertex-based blocks
    corresponding to the outgoing arcs of the inner vertices, namely
    $\outArcs{1} = \{a_1, a_2, a_3\}$, $\outArcs{2} = \{a_4, a_5\}$,
    and $\outArcs{4} = \{a_6, a_7, a_8\}$.
  }
  \label{fig:nonnested_schur_blocks}
\end{figure}

Each vertex-based block $(\Schur_{i, j})_{i, j \in \innerVertices}$ has a size
of size $\numOutCoupled{i} \times \numOutCoupled{j}$.
An illustration of the resulting block structure is shown in Figure~\ref{fig:nonnested_schur_blocks}
We choose this partitioning of the Schur complement as it provides
better performance with the subsequent methods. The following theory
holds in both settings nevertheless.
The individual blocks $\Schur_{i, j}$ can be defined in terms of
$\outArcs{i} = (a_{k_{1}} = (i, j_{1}), \ldots, a_{k_{r_{i}}} = (i,
j_{r_{i}}))$ as follows:
\begin{enumerate}
\item
  For $i = j$ the matrix $\Schur_{i, i}$ consists of blocks
  \begin{equation*}
    \begin{cases}
      \outMat{k_{l}} B_{i}^{-1} (\outMat{k_{l}})^{T} + \incMat{k_{l}} B_{j_{l}}^{-1} (\incMat{k_{l}})^T + D_{k_l} & \text{ if } l = l', \\
      \outMat{k_{l}} B_{i}^{-1} (\outMat{k_{l'}})^{T} & \text{ otherwise}, \\
    \end{cases}
  \end{equation*}
  for $l, l' = 1, \ldots, r_{i}$.
\item
  For $i \neq j$ such that $(i, j) \in \innerArcs$, we let
  $\outArcs{j} = (a_{k'_{1}}, \ldots, a_{k'_{r_{j}}})$ be the outgoing arcs of $j$,
  noting that $\Schur_{i, j}$ consists of $k_{r_{i}} \times k'_{r_{j}}$ blocks.
  The block for $l = 1, \ldots, k_{r_{i}}$ and $l' = 1, \ldots, k'_{r_{j}}$ is given by
  $\incMat{k_{l}} B_{j}^{-1} (\outMat{k'_{l'}})^{T}$ if $a_{k_{l}} = (i, j)$ and
  a zero matrix otherwise.
\item
  For $i \neq j$ with $(j, i) \in \innerArcs$, it holds that $\Schur_{j, i} = \Schur^{T}_{i, j}$.
\item
  All other blocks are zero.
\end{enumerate}
Note that while $\Schur$ does not coincide with
$\Schur_{\leq \rootVert}$, the complexity required to compute $\Schur$
is on par with the complexity of computing all Schur
complements $\Schur_{\leq i}$ using
Algorithm~\ref{alg:arrowhead_schur} which was established in
Lemma~\ref{lem:direct_running_time}.
Specifically, for each $j \in \innerVertices$ the computation of the block $\Schur_{j, j}$
requires $\numOutCoupled{j}$ solutions of systems involving $B_{j}$
plus an additional $\numInCoupled{j}$ for the block $\Schur_{i, j}$
if $(i, j) \in \Arcs$. As a result, $\mathcal{O}(l_{j})$
solutions of systems involving $B_{j}$ are required in total.

We can define the preconditioner $\Precond$ based on $\Schur$ (rather
than $-\Schur$ as was the case up to now), yielding a block-upper
triangular preconditioned matrix:
\begin{equation*}
  \Precond \define
  \begin{pmatrix}
    \mathcal{B} & \zeroMat \\
    \mathcal{C} & \Schur
  \end{pmatrix}
  \quad
  \Longrightarrow
  \quad
  \Precond^{-1}
  \begin{pmatrix}
    \mathcal{B} & \mathcal{C}^{T} \\
    \mathcal{C} & -\mathcal{D}
  \end{pmatrix}
  =
  \begin{pmatrix}
    \identMat & \mathcal{B}^{-1} \mathcal{C}^{T} \\
    \zeroMat & - \identMat
  \end{pmatrix}.
\end{equation*}
It is apparent from inspection that the preconditioned matrix has the
two eigenvalues $\pm 1$ and its minimal polynomial is $(x - 1)(x + 1)$,
leading to two-step convergence of a preconditioned GMRES method. The
application of $\Precond^{-1}$ to a right-hand side $(h, f)$,
yielding a solution of $(x, y)$, is a three-step process:
\begin{enumerate}
\item
  Compute $x$ by solving $\mathcal{B} x = h$.
\item
  Compute the right-hand side $f_{x} \gets f - \mathcal{C}x$.
\item
  Compute $y$ by solving $\Schur y = f_{x}$.
\end{enumerate}
The first step is straightforward, involving a solution of
a system involving $B_i$ for each $i \in \Vertices$.
The solution of the system involving $\Schur$ is computationally more
challenging, since $\Schur$ has a size equal to the
total number of coupled variables, therefore being
substantially larger than each of the Schur complements $\Schur_{\leq i}$
from Section~\ref{sec:direct_method}.
Of course, since $\Schur$ is
positive definite by Assumption~\ref{ass:general_assumptions}, we can
use a Cholesky decomposition to compute $y$, for example.
We can, however, alternatively employ an approach that is more tailored to the structure
of~\eqref{eq:2by2_coupled_system}. To this end, we note that the positive
definiteness of $\Schur$ implies that all of its diagonal blocks $\Schur_{i, i}$
must be positive definite as well. Consequently, the diagonal approximation
$\Schur_{\diag} \define \blkdiag (\Schur_{1, 1}, \ldots, \Schur_{\numBlocks, \numBlocks})$
of $\Schur$ is also positive definite and therefore invertible.
We can therefore use $\Schur_{\diag}$ within the fixed-point iteration
\begin{equation}
  y^{(k + 1)} \gets y^{(k)} + \Schur_{\diag}^{-1} (f_{x} - \Schur y^{(k)}),
  \label{eq:block_jacobi}
\end{equation}
based on an initial solution $y^{(0)}$.
A sufficient condition for the convergence of the iteration
is the so-called \emph{smoothing property} established by the following bound on
the spectral radius $\rho$ of the iteration matrix:
\begin{thm}
  \label{thm:spectral_radius}
  It holds that $\rho(\identMat - \Schur_{\diag}^{-1} \Schur) < 1$.
\end{thm}
\begin{proof}
  Recall that $\Schur$ as well as the blocks $\Schur_{i, i}$ and
  $\Schur_{\diag}$ are positive definite by Assumption~\ref{ass:general_assumptions}.
  Therefore, $\Schur^{1/2}_{\diag}$ and
  $\Schur^{-1/2}_{\diag}$ exist and the matrix $\identMat - \Schur_{\diag}^{-1} \Schur$
  is similar to
  \begin{equation*}
    \Schur^{1/2}_{\diag} \left( \identMat - \Schur_{\diag}^{-1} \Schur \right) \Schur^{-1/2}_{\diag} = \identMat - \Schur^{-1/2}_{\diag} \Schur \Schur^{-1/2}_{\diag}.
  \end{equation*}
  Since this matrix is symmetric, all eigenvalues of $\identMat - \Schur_{\diag}^{-1} \Schur$ are real.
  To show that all of these real eigenvalues
  are contained in $(-1, 1)$ we have to show that
  $\det((1 - \lambda) \identMat - \Schur_{\diag}^{-1} \Schur) \neq 0$ for
  all $\lambda \notin (-1, 1)$, or, equivalently that
  $\det(\mu \Schur_{\diag} - \Schur) \neq 0$ for all
  $\mu \notin (0, 2)$.

  For $\mu \leq 0$ the case is clear, since
  we add a positive multiple of a negative definite matrix to another
  negative definite matrix, preserving negative definiteness and
  therefore non-singularity.

  To handle the case $\mu \geq 2$, we show that $\hat{\Schur} \define 2 \Schur_{\diag} - \Schur$
  is positive definite: This is sufficient to ensure non-singularity
  for all $\mu > 2$, since we could then simply add a positive
  multiple of $(\mu - 2)$ of the positive definite $\Schur_{\diag}$
  to the positive definite matrix $\hat{\Schur}$.
  To show that $\hat{\Schur}$ is positive definite, consider the matrix
  $\mathcal{Z} = \blkdiag(Z_{1}, \ldots, Z_{\numBlocks})$,
  where $Z_{i} \define (-1)^{\depth(i)} \identMat$ for each
  $i \in \innerVertices$. It holds that $\mathcal{Z}^{T}\mathcal{Z} = \identMat$,
  implying that $\mathcal{Z}$ is invertible and
  $\tilde{\Schur} \define \mathcal{Z} \Schur \mathcal{Z}^{T}$ is
  positive definite. Due to the cancellation of negative signs,
  the diagonal entries of $\tilde{\Schur}$ and $\Schur$ coincide.
  Furthermore, for each arc $a_k = (i, j) \in \innerArcs$,
  we have that $(-1)^{\depth(i)} \cdot (-1)^{\depth(j)} = -1$, ensuring that
  the off-diagonal blocks in $\tilde{\Schur}$ and $\Schur$ have opposing signs.
  Consequently, we have that $\tilde{\Schur} = 2 \Schur_{\diag} - \Schur = \hat{\Schur}$,
  proving that $\hat{\Schur}$ is positive definite.
\end{proof}
\begin{cor}
  The iterates $y^{(k)}$ produced by~\eqref{eq:block_jacobi} converge
  to $y = \Schur^{-1} f_x$ for any initial $y^{(0)}$.
\end{cor}
\begin{proof}
  This is an elementary result in iterative linear algebra~(see e.g., \cite[Theorem 4.1]{iterative_methods}), directly following
  from the smoothing property established in
  Theorem~\ref{thm:spectral_radius}. \qedhere
\end{proof}
\subsection{Two-level Method}

In the following we expand on the iterative method introduced
above by applying multigrid (MG) ideas in order
to solve the Schur complement, obtaining a multi-level method.
The origins of algebraic multigrid (AMG) methods date back to the seminal
works~\cite{algebraic_multigrid_for_sparse,algebraic_multigrid},
with an introduction given in~\cite{amg_intro}.
We apply the exact two-level method~\cite[Sec. 5]{amg_intro} to the
problem of solving systems involving $\Schur$. To this end, we
let $\MLDimension_{f} \define \sum_{i \in \innerVertices} \numOutCoupled{i}$ be the dimension
of $\Schur$ and consider a \emph{coarse} approximation
of the solution space $\Real^{\MLDimension_f}$ with a dimension of $\MLDimension_{c} < \MLDimension_{f}$.
The \emph{prolongation} matrix $\MLProlong \in \Real^{\MLDimension_{f} \times \MLDimension_{c}}$ is used
to map solutions from the coarse to the original (fine) space whereas
its transpose \emph{restricts} solutions to the coarse space. The
\emph{smoothing} matrix $\MLSmoother \in \Real^{\MLDimension_{f} \times \MLDimension_{f}}$ is chosen,
as its name suggests,
to satisfy the smoothing property with respect to $\Schur$.
Furthermore, we let
$\Schur_{c} \define \MLProlong^{T} \Schur \MLProlong$ denote the
restriction of $\Schur$ to $\Real^{\MLDimension_c}$, which we assume to be invertible.
The multi-level (ML) approximation $\MLOperator \in \Real^{\MLDimension_{f} \times \MLDimension_{f}}$ of the inverse of
$\Schur$ is defined by its action on a right-hand side $g \in \Real^{\MLDimension_{f}}$,
which is computed based on the following steps:
\begin{enumerate}
\item
  Restrict $g$ to $\Real^{\MLDimension_{f}}$ via $g_{c} \gets \MLProlong^{T} g$.
\item
  Compute coarse-space correction $w_{c} \gets \Schur_{c}^{-1} g_{c}$.
\item
  Prolongate correction $w \gets \MLProlong w_{c}$.
\item
  Compute post-smoothing step $\MLOperator(g) \define w + \MLSmoother(g - \Schur w)$.
\end{enumerate}
To ensure convergence, it is once again necessary to bound the
spectral radius of the iteration matrix $\identMat - \MLOperator \Schur$:%
\begin{lem}[{\cite[Lemma 5.2]{amg_intro}}]
  \label{lem:two_level_iteration_matrix}
  The iteration matrix $(\identMat - \MLOperator \Schur)$ is given by
  \begin{equation*}
    (\identMat - \MLSmoother \Schur)(\identMat - \Pi_{c}),
  \end{equation*}
  where $\Pi_{c}$ is the $(\cdot, \cdot)_{\Schur}$-orthogonal
  projection on $\Real^{\MLDimension_{f}}$, given by $\Pi_{c} \define \MLProlong \Schur_{c}^{-1} \MLProlong^{T} \Schur$.
\end{lem}
Since the matrix $(\identMat - \Pi_{c})$ defines a $(\cdot, \cdot)_{\Schur}$-orthogonal projection, we have the estimate $\|\identMat - \Pi_{c}\|_{\Schur} \leq 1$, where $\| \cdot \|_{\Schur}$ denotes the norm induced by the corresponding scalar product. This, along with the equality $\| T \|_S = \|\Schur^{1/2} T \Schur^{-1/2}\|_2$ for any matrix $T$, allows us to bound the spectral radius of the ML iteration from above by
\begin{align*}
\rho\left(\left(\identMat - \MLSmoother \Schur\right)\left(\identMat - \Pi_{c}\right)\right) &\le \left\| \left(\identMat - \MLSmoother \Schur\right)\left(\identMat - \Pi_{c}\right) \right\|_{\Schur} \le \left\| \identMat - \MLSmoother \Schur \right\|_S \\
&= \| \identMat - \Schur^{1/2} \MLSmoother \Schur^{-1/2} \|_2 = \rho \left( \identMat - \Schur^{1/2} \MLSmoother \Schur^{-1/2} \right) = \rho\left(\identMat - \MLSmoother \Schur\right).
\end{align*}
Hence, the ML iteration itself satisfies the smoothing property with respect to $\Schur$ as long
as $\MLSmoother$ does. Since
the smoothing property holds for $\Schur_{\diag}$, setting $\MLSmoother = \Schur^{-1}_{\diag}$
is an obvious choice. Readers familiar with multigrid methods will note that this iteration
consists of a single post-smoothing step without any pre-smoothing applied. This
is for ease of notation, and in practice may be adapted to improve convergence speed.

It remains to choose a suitable coarse prolongation matrix $\MLProlong$.
An important requirement for the choice is that the
coarse-space correction $w_{c}$ can be computed efficiently. Recall that
a reason for why we cannot easily solve the matrix $\Schur$ in the first
place is the presence of the off-diagonal blocks $\Schur_{i, j}$. If
those blocks were not present, we could simply decompose the diagonal
blocks. What is more, to apply the smoother $\Schur_{\diag}^{-1}$, we may want to
decompose the diagonal blocks anyway.
We therefore propose to consider as a coarse space the restriction of
$\Schur$ to a \emph{conflict-free} set of vertices.  Two vertices
$i, j \in \innerVertices$ are said to be \emph{in conflict} if
$(i, j) \in \innerArcs$ or $(j, i) \in \innerArcs$
and \emph{conflict-free} otherwise.
A set $\innerVertices_{s}$ of vertices is \emph{conflict-free} iff all
of its vertices are conflict-free.
Notable examples of conflict-free sets
are given by the (inclusion-wise maximal) sets of even/odd vertices:
\begin{equation*}
    \innerVertices_{\even} \define \{ i \in \innerVertices \mid \depth(i) \: \mathrm{even} \} \quad \text{and} \quad
    \innerVertices_{\odd} \define \{ i \in \innerVertices \mid \depth(i) \: \mathrm{odd} \}.
\end{equation*}
To define a prolongation and restriction, consider two (ordered) subsets
$\innerVertices_{c} \subseteq \innerVertices_{f}$ of
vertices where $\innerVertices_{f} \define (i_{1}, \ldots, i_{p})$ and
$\innerVertices_{c} \define (j_{1}, \ldots, j_{q})$.
The \emph{subset operator} $\MLSubset_{\innerVertices_{c}, \innerVertices_{f}}$ mapping from $\Real^{\numOutCoupled{i_{1}}} \times \cdots \times \Real^{\numOutCoupled{i_{p}}}$
to $\Real^{\numOutCoupled{{j_{1}}}} \times \cdots \times \Real^{\numOutCoupled{{j_{q}}}}$
is defined by its action on a vector
$(y_{i_{1}}, \ldots, y_{i_{p}})$ as
\begin{equation*}
  \MLSubset_{\innerVertices_{c}, \innerVertices_{f}}((y_{i_{1}}, \ldots, y_{i_{p}})) \define
  (y_{j_{1}}, \ldots, y_{j_{q}}).
\end{equation*}
For a given conflict-free set $\innerVertices_{s} = \{i_{{1}}, \ldots, i_{{l}}\}$, we can therefore obtain
a two-level method by defining the prolongation matrix to
be $\MLProlong \define \MLSubset^{T}_{\innerVertices_{s}, \innerVertices}$.
As a consequence, the restriction $\MLProlong^{T} = \MLSubset_{\innerVertices_{s}, \innerVertices}$ removes from a vector $y$
the entries not belonging to $\innerVertices_s$, and
\begin{equation*}
  \Schur_{c} = \blkdiag(\Schur_{{1}, {1}}, \ldots, \Schur_{{l}, {l}})
\end{equation*}
is block-diagonal and invertible, enabling the efficient computation of the coarse-space correction.

\subsection{Multi-level Methods}
\label{sec:multi_level}

A different choice for the prolongation matrix $P$ may be used
to derive multi-level methods aiding in solving for the Schur
complement $\Schur$. A \emph{multi-level method} with $\MLNumLevels$
levels consists of $\MLNumLevels$ increasingly fine approximations
of dimensions
$\MLDimension_{c} = \MLDimension_{1} \leq \MLDimension_{2} \leq \ldots \leq \MLDimension_{f}$
of the original solution
space together with smoothing and prolongation matrices
$\MLSmoother_{j} \in \Real^{\MLDimension_{j} \times \MLDimension_{j}}$,
$\MLProlong_{j} \in \Real^{\MLDimension_{j + 1} \times \MLDimension_{j}}$ yielding restrictions
$\Schur_{j} \define \MLProlong_{j}^{T} \cdots \MLProlong_{\MLNumLevels - 1}^{T} \Schur \MLProlong_{\MLNumLevels - 1}
\cdots \MLProlong_{j}$ for $j = 1, \ldots, \MLNumLevels - 1$, where we set $\Schur_{1} \enifed \Schur_{c}$.
The \emph{V-cycle} consists of computing corrections on increasingly
coarse subspaces up to $\Real^{\MLDimension_{f}}$ (where $\Schur_{c}$ is solved exactly), followed
by a post-smoothing step.
The corresponding ML matrix $\MLOperator_{j} \in \Real^{\MLDimension_{j} \times \MLDimension_{j}}$ at level $j$ can
be described recursively by its action on a right-hand side $g_{j} \in \Real^{\MLDimension_{j}}$:
\begin{enumerate}
\item
  If $j = 1$, return $\Schur_{c}^{-1} g_{j}$.
\item
  Restrict $g_{j}$ to $\Real^{\MLDimension_{j - 1}}$ via $g_{j - 1} \gets \MLProlong_{j}^{T} g_{j}$.
\item
  Compute coarse-space correction $w_{j - 1} \gets \MLOperator_{j - 1}(g_{j - 1})$.
\item
  Prolongate correction $w_{j} \gets \MLProlong_{j - 1} w_{j - 1}$.
\item
  Compute post-smoothing step $\MLOperator_{j}(g_{j}) \define w_{j} + \MLSmoother_{j}(g_{j} - \Schur_{j} w_{j})$.
\end{enumerate}
To define the matrices, we consider a nested family of sets of vertices
$\innerVertices_{c} = \innerVertices_{1} \subset \cdots \subset \innerVertices_{\MLNumLevels} = \innerVertices$.
For each level $j = 1, \ldots, \MLNumLevels$ we pick as $\Real^{\MLDimension_{j}}$ the space
associated with the variables corresponding to $\innerVertices_{j}$.
Correspondingly, the prolongation matrix for
$j \in \{1, \ldots, \MLNumLevels - 1\}$ is given by
$\MLProlong_{j} \define \MLSubset^{T}_{\innerVertices_{j}, \innerVertices_{j + 1}}$ Furthermore, we let
the finest smoothing matrix be the original diagonal part of $\Schur$,
\ie $\MLSmoother_{\MLNumLevels} \define \Schur_{\diag}^{-1}$, and set
$\MLSmoother_{j} \define \MLProlong_{j}^{T} \cdots \MLProlong_{\MLNumLevels - 1}^{T} \MLSmoother_{\MLNumLevels} \MLProlong_{\MLNumLevels - 1} \cdots
\MLProlong_{j}$, which is to say the non-singular matrix consisting of the blocks of
$\Schur_{\diag}^{-1}$ corresponding to $\innerVertices_{j}$. Finally, to ensure
that we can solve $\Schur_{c}$ efficiently, we ask that $\innerVertices_{c}$ be
conflict-free. To define nested families of sets of arcs, we once
again make use of the topology of the graph by setting
\begin{equation*}
  \begin{aligned}
    \innerVertices_{\leq k} &\define \{ i \in \innerVertices \mid \depth(i) \leq k \} \text{, and} \\
    \innerVertices_{\geq k} &\define \{ i \in \innerVertices \mid \depth(i) \geq k \},
  \end{aligned}
\end{equation*}
for $k = 1, \ldots, \height(\innerGraph)$. Clearly it holds that
$\innerVertices_{\leq j} \subset \innerVertices_{\leq j + 1}$, and in particular
$\innerVertices_{\leq \height(\innerGraph)} = \innerVertices$. Furthermore, the set
$\innerVertices_{\leq 1} = \{\rootVert\}$ is conflict-free. Thus, we
can construct a multi-level method with $\MLNumLevels = \height(\innerGraph)$ levels based
on these arc sets by setting $\innerVertices_{j} \define \innerVertices_{\leq j}$, where
the coarsest solution space consists of the arcs incident to the root
$\rootVert$. Symmetrically, we can consider the sets
$\innerVertices_{\geq j + 1} \subset \innerVertices_{\geq j}$, where
$\innerVertices_{\geq 1} = \innerVertices$ and $\innerVertices_{\geq \height(\innerGraph)}$ consist of
the (conflict-free) set of leaves of $\innerGraph$. Thus, we
can set $\innerVertices_{j} \define \innerVertices_{\geq \height(\Graph) - j + 1}$ to
obtain another multi-level method. We call these methods
the \emph{Bottom-Up} and \emph{Top-Down} multi-level method, respectively.

As for the convergence of these methods, we can
apply Lemma~\ref{lem:two_level_iteration_matrix} repeatedly going
from the coarsest space (where the restricted system is solved exactly)
to the finest space. To this end, the smoothing matrix
$\MLSmoother_{j}$ has to satisfy the smoothing property with
respect to $\Schur_{j}$ for all $j = 2, \ldots, \MLNumLevels$.
It is, however, easy to see that
these smoothing properties are always satisfied, since they correspond
to the original smoothing property established in
Theorem~\ref{thm:spectral_radius}, applied either to a single subtree in the
Bottom-Up or several times to different subtrees in the Top-Down method.

Finally, several alternative iteration schemes, known as W- and
F-cycles~\cite{multigrid_tut} have been introduced and can be easily
adapted into corresponding multi-level methods. The notable difference
between V- and W-cycle consists of the fact that the latter iteration
performs two coarse-correction steps rather than a single one in each
level. Consequently, the running time of the W-cycles grows
exponentially in $\MLNumLevels = \height(\Graph)$, whereas the running
times of V- and F-cycles remain polynomial in $\MLNumLevels$.
\subsection{Super-Node Smoothing}
\label{sec:super_node}
The matrix $\Schur_{\diag}$ introduced above has the smoothing
property and is therefore a suitable choice as a smoother, while
having the disadvantage of not capturing the adjacency in $\innerGraph$.
An alternative is to incorporate the adjacency of a subset of the
vertices in $\innerVertices$. To capture some of this information,
we can merge a subset of the vertices in $\Graph$ into
a single super-vertex. Specifically, we let $i \in \innerVertices$ and
$\outArcs{i} = (a_{k_{1}} = (i, j_{1}), \ldots, a_{k_{r_{i}}} = (i, j_{r_{i}}))$
and treat the submatrix of $\Schur$ composed of the blocks for
$i, j_1, \ldots, j_{r_{i}}$ as a single larger block.
In terms of the underlying graph structure, the creation of
the super-vertex is equivalent to the contraction of the
arcs $a_{k_{1}}, \ldots, a_{k_{r_{i}}}$ in any order, resulting
in a smaller directed tree corresponding to the change in the blocks
of $\Schur$.
Consequently, Theorem~\ref{thm:spectral_radius}
still applies and we obtain another smoother to be used in our multi-level
method. However, to apply the inverse of this
\emph{super-node} smoother to a right-hand side, we have to solve
the system corresponding to the newly created super-vertex, possibly requiring
an additional factorization.
The approach can be extended by creating multiple super vertices based
on different vertices. Furthermore, given a subset
$\innerVertices_j \subseteq \innerVertices$ associated with a specific level $j$
in a multi-level method, we can pick these super vertices depending on
$\innerVertices_j$ (provided that $i, j_{1}, \ldots, j_{r_{i}}$ are
contained in $\innerVertices_j$) in order to obtain a level-specific
smoother.  Since the intuitive importance of arcs decreases with
decreasing height, we propose to turn all vertices in
$\innerVertices_{j}$ with maximum height into super vertices,
yielding the smoother $\Schur_{\super}(\innerVertices_{j})$.

\section{Computational Experiments}
\label{sec:experiments}

In this section, we showcase the performance of the presented preconditioners with a set of numerical experiments. We examine the convergence of preconditioned Krylov solvers and their dependence on the problem parameters and preconditioner settings. Table~\ref{tab:preconditioners_overview} provides an overview of the preconditioners that are analyzed subsequently, along with their abbreviated identifiers and definitions. The tests are implemented in \texttt{Python}~3.9.18 \cite{python3}, complemented with the libraries \texttt{SciPy}~1.10.0 \cite{scipy} for various numerical linear algebra routines, and \texttt{DOLFINx} for finite element method functionalities \cite{dolfinx}. All experiments were conducted on an Intel\textsuperscript{\textregistered} Core\textsuperscript{TM} i7-10710U processor.

The following experiments serve as a proof of concept regarding the convergence behavior and the scaling of the computational cost of the iterative methods. We showcase our ability to compete with widely used methods by considering a specific example in Section~\ref{sec:multiple_shooting}. It is noted that most of the following problems can be solved efficiently with direct linear solvers, including sparse LU factorizations such as SuperLU~\cite{superlu}. However, our methods demonstrate better scaling and, therefore, can enable an improvement over direct methods for large problem sizes, so a specific parallel implementation to solve problems of huge scale is a key avenue of future work.

\begin{table}
  \centering
  \begin{tabular*}{\textwidth}{@{\extracolsep\fill}ll@{}}
    \toprule
    \textbf{Preconditioner} & \textbf{Definition} \\
    \midrule
    Nested Block-diagonal (Nested B-Diag) & Eq.~\eqref{eq:block_diagonal_preconditioner} in Section~\ref{sec:nested_block_diagonal} \\
    Nested Hook & Eq.~\eqref{eq:hook_preconditioner} in Section~\ref{sec:nested_hook} \\
    Nested Recursive & Eq.~\eqref{eq:recursive_preconditioner} in Section~\ref{sec:nested_recursive} \\
    Nested Exact & Section~\ref{sec:nested_exact} \\
    Non-nested Block-diagonal (Non-nested B-Diag) & Eq.~\eqref{eq:block_jacobi} in Section~\ref{sec:direct_preconds} \\
    Multi-level (ML) & Section~\ref{sec:multi_level} \\
    \bottomrule
  \end{tabular*}
  \caption{Overview of the tested preconditioners and their definitions.}
  \label{tab:preconditioners_overview}
\end{table}

\subsection{Scenario Tree NMPC}
\label{sec:scenario_mpc}
The first test problem stems from optimal control of systems under uncertainty. A well-established method for robustly tackling (deterministic) optimal control problems is {\it nonlinear model predictive control} (NMPC, see \cite{nmpc}). Uncertainties of systems can be modeled by scenario trees, which, when integrated with NMPC, leads to so-called {\it scenario tree NMPC} (see e.g., \cite{scenario_mpc}). How scenario tree NMPC can be applied in practice is presented in \cite{scenario_mpc_software}. The control problem is reduced to an optimization problem which leads to linear systems fitting into the framework discussed here. We consider the example provided in \cite{scenario_mpc_software}, which models masses coupled through spring packets containing redundant arrays of springs, with each spring having a certain fault probability.

\begin{figure}
\centering
\setlength{\figureheight}{0.6\textwidth}
\setlength{\figurewidth}{\textwidth}
\includegraphics[width=\textwidth]{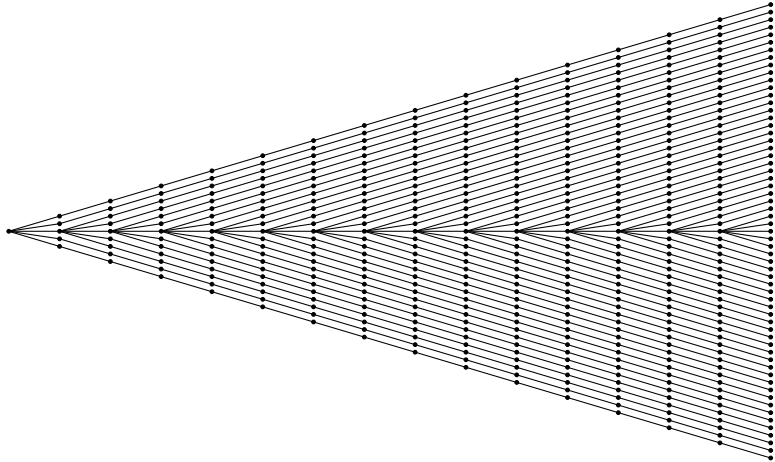}
\caption{Tree resulting from the application of the scenario tree NMPC approach \cite{scenario_mpc_software} to the problem of connected spring packets. Each level corresponds to a decision point of the scenarios. Each path from the root to a leaf represents a single scenario. The preconditioners are based on the topology of this scenario tree.}
\label{fig:scenario_mpc_pruned_tree}
\end{figure}

\begin{figure}
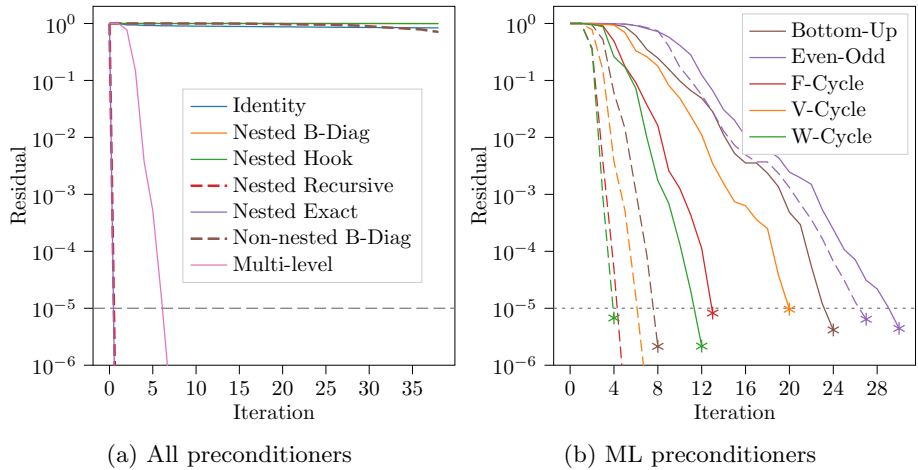

  \centering
  \begin{subfigure}{0.49\textwidth}
    \setlength{\figureheight}{\textwidth}
    \setlength{\figurewidth}{\textwidth}
    \includegraphics[width=\textwidth]{scenario_mpc_all}
    \caption{All preconditioners}
    \label{fig:benchmark_scenario_mpc:all}
  \end{subfigure}
  \begin{subfigure}{0.49\textwidth}
    \setlength{\figureheight}{\textwidth}
    \setlength{\figurewidth}{\textwidth}
    \includegraphics[width=\textwidth]{scenario_mpc_ml}
    \caption{ML preconditioners}
    \label{fig:benchmark_scenario_mpc:ml}
  \end{subfigure}
\caption{Convergence for the scenario tree NMPC problem with different preconditioners for a tree depth of $15$. The plots show the relative residual of the solution at each GMRES iteration: all preconditioners with one representative ML preconditioner (left), various ML preconditioners (right). The right plot shows the convergence for the block-diagonal (solid lines) and the super-node smoother (dashed lines). The nested recursive, nested exact, and ML preconditioners show fast convergence, while the others do not provide a significant improvement over the unpreconditioned method.}
\label{fig:benchmark_scenario_mpc}
\end{figure}

We fix the time step for the scenario tree generation to \num{0.1}. The underlying system contains \num{4} spring packets each of which contains \num{6} springs. The spring fault probability is set to \num{0.003}. The scenario tree is chosen to model \num{15} time steps, corresponding to a tree of depth \num{15}. The resulting scenario tree is depicted in Figure~\ref{fig:scenario_mpc_pruned_tree}. The overall system has \num{12432} degrees of freedom and \num{76} nodes. The resulting linear system is then solved with GMRES equipped with the preconditioners introduced above. The iterative solver is run without restarts and a maximum of \num{100} iterations. To allow a better comparison of various preconditioners, we use right preconditioning. That is, to solve $\mathcal{A} x = b$, GMRES is applied to the system
\begin{align*}
\widetilde{\mathcal{A}} \widetilde{x} = b\, \quad \text{with} \quad \widetilde{\mathcal{A}} = \mathcal{A} \mathcal{P}^{-T} \text{ and } x = \mathcal{P}^{-T} \widetilde{x}.
\end{align*}
It is noted that the matrices $\mathcal{P}^{-1}\mathcal{A}$ and $\mathcal{A}\mathcal{P}^{-T}$ are spectrally equivalent and possess the same minimal polynomial. Thus, the above theoretical results also apply to right preconditioning with the transpose.
Figure~\ref{fig:benchmark_scenario_mpc} shows the number of GMRES iterations for different preconditioners and the 2-norm of the associated relative residuals
\begin{equation*}
r_{r, k} \define \frac{r_k}{\left\| r_0 \right\|},
\end{equation*}
where $r_k$ denotes the absolute residual at the $k$th iteration and $r_0 = b$ if we start with a zero initial guess. As found in experiments, different settings for the ML preconditioner lead to similar convergence results. Thus, Figure~\ref{fig:benchmark_scenario_mpc:all} only shows one representative for the ML preconditioner (the V-cycle with super-node smoother introduced in Section~\ref{sec:super_node}), while Figure~\ref{fig:benchmark_scenario_mpc:ml} provides a detailed insight into the behavior of different settings of the ML method. 

As expected, the nested exact and recursive preconditioners lead to convergence within one GMRES iteration. This aligns with the results of Lemma~\ref{lem:exact_recursive_solution} and Corollary~\ref{cor:exact_convergence}. The nested block-diagonal preconditioner does not result in a significant reduction in the iteration numbers. The nested hook preconditioner slows down the convergence and was found to be effective only for shallow trees.
The discussion of the non-nested block-diagonal preconditioner refers to \eqref{eq:block_jacobi}, and the experiment suggests that the preconditioner does not resemble the system well enough to serve as an effective preconditioner. Finally, we have the ML preconditioner, giving a significant reduction in the number of iterations. Figure~\ref{fig:benchmark_scenario_mpc:ml} depicts variations of the ML approach (indicated by color) paired with either the block-diagonal smoother (solid lines) or the super-node smoother (dashed lines). The super-node smoother performs better than the block-diagonal smoother in general.

When increasing the tree depth to \num{20}, most preconditioners failed due to excessive runtime or the exhaustion of the maximum number of iterations. On the one hand, this failure can be explained by the increased runtime for each application of some of the preconditioners due to their exponentially increasing complexity with respect to the tree depth. While the nested block-diagonal, the nested hook, and the non-nested block-diagonal preconditioners do not have such runtime scaling, these do not provide a sufficient performance in order to converge within the prescribed maximum number of iterations. The ML preconditioner using a V-cycle, however, still showed good convergence behavior, yielding convergence within \num{9} iterations in the case of the super-node smoother and within \num{26} iterations in the case of the block-diagonal smoother.

\subsection{Multiple Shooting for Optimal Control}
\label{sec:multiple_shooting}

We consider the solution of optimal control problems with ordinary differential equations, more specifically, problems of the form
\begin{alignat*}{3}
\min_{y, u} &\quad \frac{1}{2}\int_0^T \left\| y(t) - y_d(t) \right\|^2\mathrm{d}t + \frac{\beta}{2}\int_0^T \left\| u(t) - \bar{u} \right\|^2 \, \mathrm{d}t, \span \span\\
\quad \text{s.t.} &\quad \dot{y}(t) &&= f(y(t), u(t)), \\
&\quad y(0) &&= y_0.
\end{alignat*}
When considering large time horizons, such problems can lead to numerically challenging problems since the optimality conditions lead to equations both forward and backward in time. Parallel-in-time methods (see e.g., \cite{50years_pint}) tackle this problem by allowing parallelization along the time axis. Among various types of such methods, we focus on a {\it multiple shooting} method \cite{multiple_shooting, multiple_shooting_bock_plitt}. The time interval $[0, T]$ is separated into smaller intervals $[t_k, t_{k + 1}]$ and, where necessary, appropriate matching constraints are enforced at the interfaces. This results in a series of subsystems, each of which can be considered independently. The subintervals can be arranged into sparsely coupled trees (see Figure~\ref{fig:multiple_shooting}). This tree structure is not just restricted to binary trees but can be generalized to arbitrary numbers of children. 

\begin{figure}
  \centering
  \includegraphics[width=\textwidth]{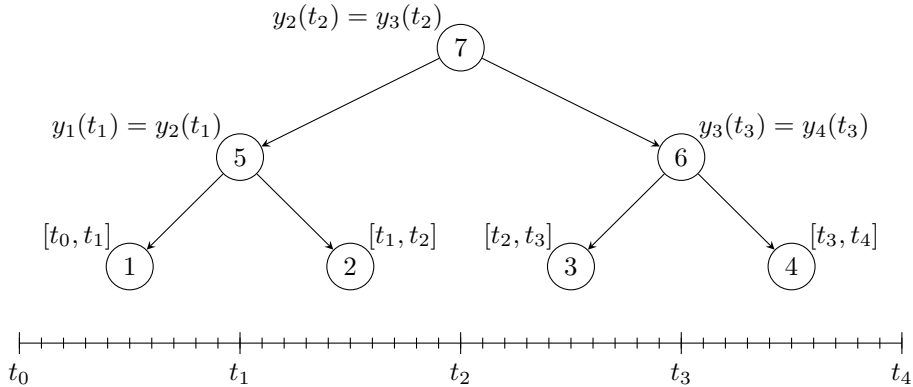}
  \caption{Rearranging matching constraints in multiple shooting.}
  \label{fig:multiple_shooting}
\end{figure}  

\begin{figure}
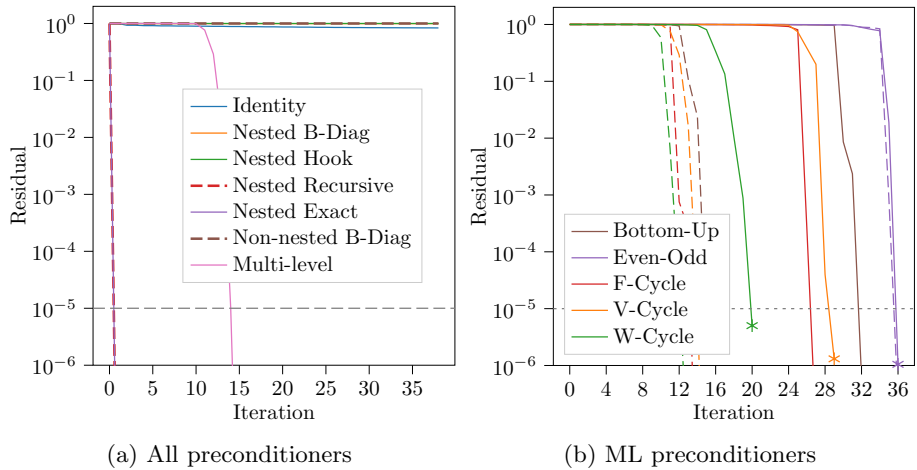

  \centering
  \begin{subfigure}{0.49\textwidth}
    \setlength{\figureheight}{\textwidth}
    \setlength{\figurewidth}{\textwidth}
    \includegraphics[width=\textwidth]{multiple_shooting_binary_all}
    \caption{All preconditioners}
    \label{fig:benchmark_multiple_shooting_binary:all}
  \end{subfigure}
  \begin{subfigure}{0.49\textwidth}
    \setlength{\figureheight}{\textwidth}
    \setlength{\figurewidth}{\textwidth}
    \includegraphics[width=\textwidth]{multiple_shooting_binary_ml}
    \caption{ML preconditioners}
    \label{fig:benchmark_multiple_shooting_binary:ml}
  \end{subfigure}
\caption{Convergence for the multiple shooting problem with a binary tree. The plots follow the same structure as in Figure~\ref{fig:benchmark_scenario_mpc}. The nested recursive, nested exact, and ML preconditioners admit fast convergence, while the others do not provide a significant improvement over the unpreconditioned method.}
\label{fig:benchmark_multiple_shooting_binary}
\end{figure}

We look at a modified Lotka--Volterra problem
\begin{align*}
f(y, u) = \begin{pmatrix}
y_1 - y_1 y_2 - c_1 y_1 u \\
-y_2 + y_1 y_2 - c_2 y_2 u
\end{pmatrix},
\end{align*}
with $y_d = (1, 1)^T$, $\bar{u} = 0.5$, $T = 12$, $c_1=0.4$, $c_2=0.2$, $\beta=0.1$, and the initial state $y_0 = (0.5, 0.7)^T$. First, the time horizon is separated into \num{8} subintervals and arranged in a binary tree with depth \num{4}. Figure~\ref{fig:benchmark_multiple_shooting_binary} shows the performance of the preconditioners for this setting. We can, in fact, observe similar behavior of the preconditioners as in the scenario tree NMPC problem except for the ML preconditioner. The first few iterations of the ML preconditioners do not lead to a significant improvement in the residual, but it eventually achieves rapid convergence.

\begin{figure}
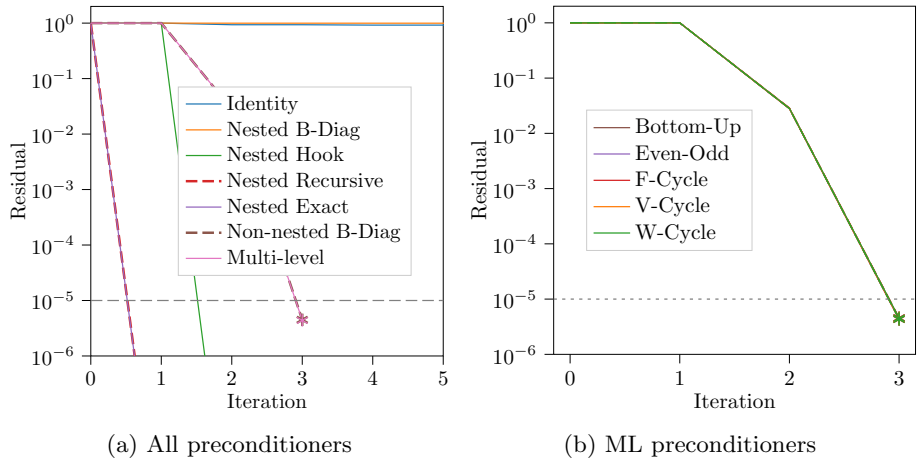

  \centering
  \begin{subfigure}{0.49\textwidth}
    \setlength{\figureheight}{\textwidth}
    \setlength{\figurewidth}{\textwidth}
    \includegraphics[width=\textwidth]{multiple_shooting_shallow_all}
    \caption{All preconditioners}
    \label{fig:benchmark_multiple_shooting_shallow:all}
  \end{subfigure}
  \begin{subfigure}{0.49\textwidth}
    \setlength{\figureheight}{0.97\textwidth}
    \setlength{\figurewidth}{\textwidth}
    \includegraphics[width=\textwidth]{multiple_shooting_shallow_ml}
    \caption{ML preconditioners}
    \label{fig:benchmark_multiple_shooting_shallow:ml}
  \end{subfigure}
\caption{Convergence for the multiple shooting problem with a shallow tree. The plots follow the same structure as in Figure~\ref{fig:benchmark_scenario_mpc}. All preconditioners admit fast convergence except for the nested block-diagonal preconditioner.}
\label{fig:benchmark_multiple_shooting_shallow}
\end{figure}

In a second setting, we consider a shallow tree with the root node having \num{64} children. The results are depicted in Figure~\ref{fig:benchmark_multiple_shooting_shallow}. In contrast to the preceding experiments, we can observe that the hook preconditioner convergences within two iterations, which aligns with Corollary~\ref{cor:hook_convergence}. All ML preconditioners are found to converge within \num{3} GMRES iterations. The non-nested block-diagonal preconditioner shows the same behavior as the V-cycle. The remaining convergence results overlap with what we have seen in the previous experiments.

\paragraph{Comparison with condensing.}
Finally, we compare the iterative solvers we have derived to an established problem-specific method, called {\it condensing}~(cf.~\cite{multiple_shooting_bock_plitt} and \cite[p.~560 ff.]{mpc}). The condensing approach recursively constructs an explicit expression for the state as a function of the control, making use of the block-sparse structure of the underlying matrices. This way, the state variable is eliminated in the optimal control problem, and the problem is cast as an equivalent unconstrained quadratic program. While this is usually carried out on the shooting grid, it is more reasonable here to perform the condensing on the (much finer) time discretization grid due to the discretization of the control, which aligns with the time grid and not with the multiple shooting grid. The method is known to scale quadratically in the number of timesteps for constructing the explicit form of the state. Even though the resulting quadratic program is dense and the runtime of its solution would scale cubically, it is described in \cite{mpc} how a Cholesky decomposition of the resulting linear system can be constructed with linear runtime scaling, reducing the method's overall runtime behavior to quadratic scaling. In the following analysis, the efficient computation of the Cholesky decomposition is not implemented. Since the runtime of condensing is dominated by the construction of the necessary dense matrices anyway, it is sufficient to measure the runtime of this step only to get an insight into how the methods compare.

Figure~\ref{fig:benchmark_multiple_shooting_condensing} compares the runtimes of condensing and the multiple shooting approach coupled with our ML-preconditioned iterative solver. The ML preconditioner is configured to use a V-cycle with the block-diagonal smoother. The domain is separated into sub-intervals for multiple shooting, each containing \num{40} equidistant forward Euler timesteps, and is arranged into a shallow tree. The $x$-axis shows the number of timesteps, and the $y$-axis denotes the runtime in seconds for assembling the necessary matrices and solving the system (except for the condensing method).
The dashed lines represent linear and quadratic scaling. The condensing approach has a quadratic scaling, consistent with the theoretical results. In contrast, the preconditioned iterative method exhibits a linear scaling and outperforms condensing for larger problem sizes. Additionally, the ML approach has the potential for parallelization, which would result in an even greater speedup. The same asymptotic scaling would be expected if the condensing was carried out on the coarse shooting grid.


\begin{figure}
\centering
\setlength{\figureheight}{0.49\textwidth}
\setlength{\figurewidth}{0.49\textwidth}
\includegraphics[width=0.49\textwidth]{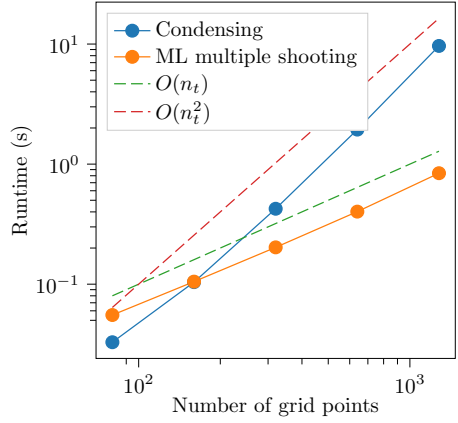}
\caption{Runtime comparison of condensing and the ML preconditioner for solving the multiple shooting problem for various numbers of timesteps~$n_t$. The linear runtime behavior of the ML-preconditioned solver provides better scaling than the quadratic scaling of the condensing approach.}
\label{fig:benchmark_multiple_shooting_condensing}
\end{figure}

\subsection{Domain Decomposition for PDEs}

\begin{figure}
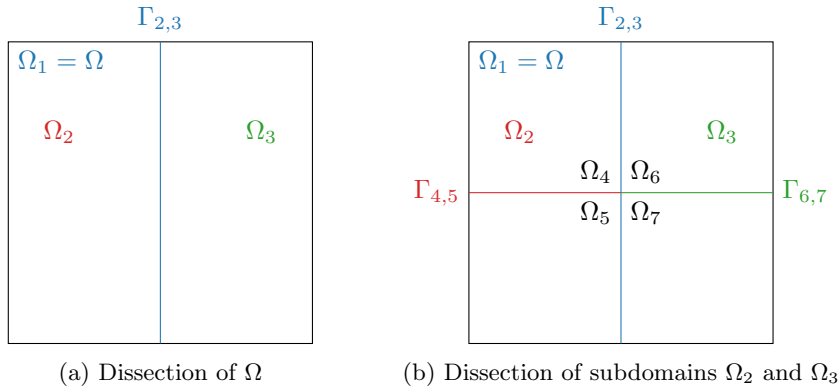

  \centering
    \begin{subfigure}{0.49\textwidth} 
        \centering
        \includegraphics{domain_decomposition_1}
        \caption{Dissection of $\Omega$}
    \end{subfigure}
    \begin{subfigure}{0.49\textwidth} 
        \centering
        \includegraphics{domain_decomposition_2}
        \caption{Dissection of subdomains $\Omega_2$ and $\Omega_3$}
    \end{subfigure}
    \caption{Construction of subdomains by recursive dissection of the domain $\Omega$. The dissecting lines of two domains $\Omega_i$ and $\Omega_j$ are labeled by $\Gamma_{i,j}$, along which consensus constraints are enforced.}
    \label{fig:dom_decomp_pde}
\end{figure}

The tree-sparse approach to multiple shooting in Section \ref{sec:multiple_shooting} can be transferred to domain decomposition methods for elliptic PDEs. Given a PDE on a domain $\Omega$, we can dissect $\Omega$ into an arbitrary number of subdomains. This allows us to solve the PDE on each subdomain separately while enforcing consensus constraints along the interfaces, also known as non-overlapping domain decomposition \cite[p.~74~ff.]{domain_decomposition}. This procedure can be repeated for all the subdomains recursively, resulting in a hierarchical decomposition of the domain, which can then be arranged into a tree structure (see Figure~\ref{fig:dom_decomp_pde}). The novel aspect of this approach is that the interfaces are interconnected across a tree, resulting in a more structured Schur complement. While the primary challenge of non-overlapping domain decomposition lies in approximating the Schur complement, we assume that complete information is available and demonstrate that the tree-sparse ordering coupled with our methods still provides good performance. This encourages future work to seek more efficient approximations of the structured and simpler Schur complement and to make the approach practicable. For the experiment, we consider the mixed formulation of the Poisson problem (see e.g., \cite{poisson_dom_decomp}). The solution of the Poisson problem within this formulation can be viewed as an optimization problem, hence allowing us to fit these into the above framework. A more detailed discussion of the theoretical background can be found in Appendix~\ref{app:dom_decomp}.

\begin{figure}
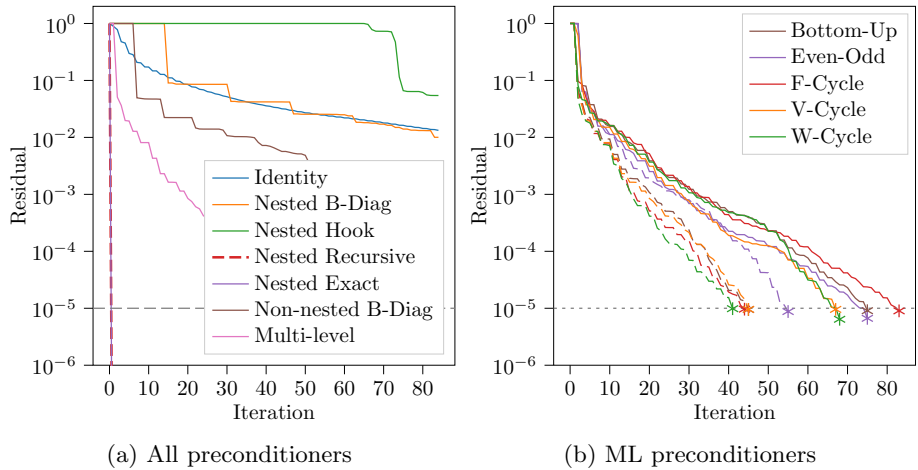

  \centering
  \begin{subfigure}{0.49\textwidth}
    \setlength{\figureheight}{\textwidth}
    \setlength{\figurewidth}{\textwidth}
    \includegraphics[width=\textwidth]{dom_decomp_pde_all}
    \caption{All preconditioners}
    \label{fig:benchmark_domain_decomp_pde:all}
  \end{subfigure}
  \begin{subfigure}{0.49\textwidth}
    \setlength{\figureheight}{\textwidth}
    \setlength{\figurewidth}{\textwidth}
    \includegraphics[width=\textwidth]{dom_decomp_pde_ml}
    \caption{ML preconditioners}
    \label{fig:benchmark_domain_decomp_pde:ml}
  \end{subfigure}
\caption{Convergence for the domain decomposition problem arranged in a binary tree. The plots follow the same structure as in Figure~\ref{fig:benchmark_scenario_mpc}. The nested recursive, nested exact, and ML preconditioners show fast convergence. The non-nested block-diagonal preconditioner provides slightly improved convergence compared to the unpreconditioned method.}
\label{fig:benchmark_domain_decomp_pde}
\end{figure}

We discretize the solution with a non-uniform triangulation of the domain $\Omega = [0, 1]^2$ and a maximum cell size of $h = \num{0.0125}$. The discretization conforms with the predefined interfaces. In our first setting, the domain is decomposed into eight subdomains, which are then arranged in a binary tree of depth four. The convergence of the preconditioners is shown in Figure~\ref{fig:benchmark_domain_decomp_pde} following the same structure as before. The nested exact and recursive preconditioners converge within one iteration. The hook preconditioner is again found to be ineffective for deep trees. The nested block-diagonal preconditioner does not result in improved convergence when compared to the identity preconditioner. As opposed to the previous experiments, the non-nested block-diagonal preconditioner leads to a reduction in iteration numbers, but is still outperformed by the ML preconditioners.

\begin{figure}
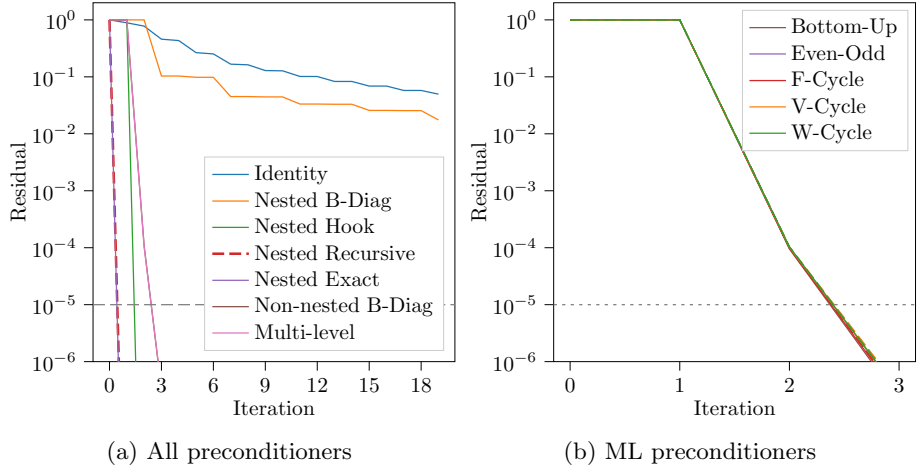

  \centering
  \begin{subfigure}{0.49\textwidth}
    \setlength{\figureheight}{\textwidth}
    \setlength{\figurewidth}{\textwidth}
    \includegraphics[width=\textwidth]{dom_decomp_pde_local_level_all}
    \caption{All preconditioners}
    \label{fig:benchmark_domain_decomp_pde_local_levels:all}
  \end{subfigure}
  \begin{subfigure}{0.49\textwidth}
    \setlength{\figureheight}{\textwidth}
    \setlength{\figurewidth}{\textwidth}
    \includegraphics[width=\textwidth]{dom_decomp_pde_local_level_ml}
    \caption{ML preconditioners}
    \label{fig:benchmark_domain_decomp_pde_local_levels:ml}
  \end{subfigure}
\caption{Convergence for the domain decomposition problem arranged in a tree with each inner node having \num{8} children. The plots follow the same structure as in Figure~\ref{fig:benchmark_scenario_mpc}. All preconditioners admit fast convergence except for the nested block-diagonal preconditioner.}
\label{fig:benchmark_domain_decomp_pde_local_levels}
\end{figure}

Similarly to Section~\ref{sec:multiple_shooting}, the tree can be arranged as a shallow tree with \num{8} leaves directly connected to the root. In Figure~\ref{fig:benchmark_domain_decomp_pde_local_levels}, we see the same behavior as in the case of the shallow tree in the multiple shooting test problem.

\begin{figure}
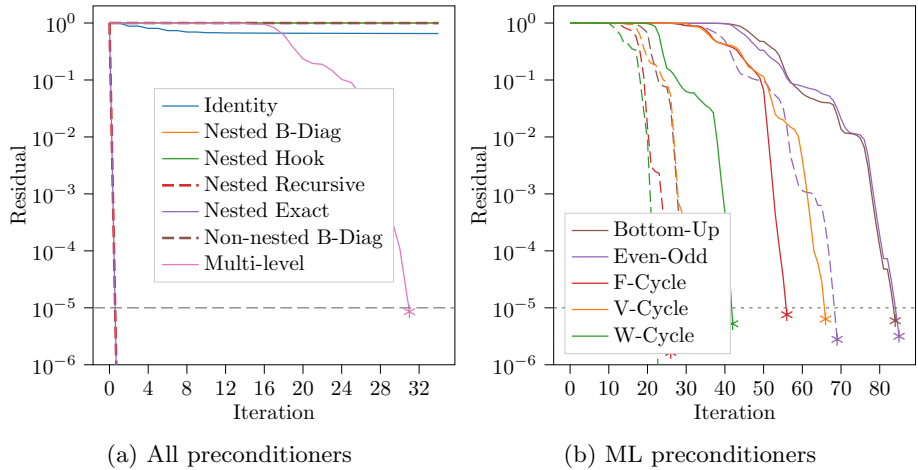

  \centering
  \begin{subfigure}{0.49\textwidth}
    \setlength{\figureheight}{\textwidth}
    \setlength{\figurewidth}{\textwidth}
    \includegraphics[width=\textwidth]{dom_decomp_pcop_all}
    \caption{All preconditioners}
    \label{fig:benchmark_domain_decomp_pcop:all}
  \end{subfigure}
  \begin{subfigure}{0.49\textwidth}
    \setlength{\figureheight}{\textwidth}
    \setlength{\figurewidth}{\textwidth}
    \includegraphics[width=\textwidth]{dom_decomp_pcop_ml}
    \caption{ML preconditioners}
    \label{fig:benchmark_domain_decomp_pcop:ml}
  \end{subfigure}
\caption{Convergence for the domain decomposition problem for a PDE-constrained optimization problem arranged in a binary tree. The plots follow the same structure as in Figure~\ref{fig:benchmark_scenario_mpc}. The nested recursive, nested exact, and ML preconditioners show fast convergence.}
\label{fig:benchmark_domain_decomp_pcop}
\end{figure}

This approach can be extended to PDE-constrained optimization problems. For an overview of such problems, the reader is referred to \cite{pcop}. We consider a Poisson control problem with a distributed control, i.e.,
\begin{alignat*}{5}
\min_{y, u} &&\quad \frac{1}{2} \int_{\Omega} \left| y(x) - y_d(x) \right|^2 \mathrm{d}x + \frac{\beta}{2} \int_{\Omega} u(x)^2 \mathrm{d}x, \span \span \span \\
\text{s.t.} &&\quad -\Delta y(x) &= u(x) \quad &&\text{in } \Omega, \\
&&\quad y(x) &= 0 \quad &&\text{on } \partial \Omega,
\end{alignat*}
where $u$ denotes a distributed control and $\beta > 0$ is a regularization parameter. The domain decomposition is carried out the same way as for the Poisson problem.

Again, we use a triangulation of the domain $\Omega = [0, 1]^2$ with a maximum cell size of \num{0.0125}, and the regularization parameter is set to $\beta=\num{1e-4}$. The desired state $y_d$ is set to
\begin{align*}
y_d(x) = \sin\left(\pi x_1\right) \sin\left(\pi x_2\right).
\end{align*}
If we arrange the domain into a binary tree of depth four, we obtain the convergence behavior depicted in Figure~\ref{fig:benchmark_domain_decomp_pcop}. The convergence behavior aligns with what has been observed in the previous examples with deep trees. For the ML preconditioners, the residual first remains on a plateau, but eventually achieves rapid convergence.  



\subsection{Computational Cost}
\begin{table}[t]
  \centering
  \begin{tabular*}{\textwidth}{@{\extracolsep\fill}lllS[table-format=4]S[table-format=5]S[table-format=6]@{}}
    \toprule
    \multicolumn{3}{c}{\textbf{Preconditioner}} & \multicolumn{3}{c}{\textbf{Tree Depth}} \\
    \cmidrule{4-6}
    & & & 5 & 10 & 15 \\
    \midrule
    \multicolumn{3}{l}{\textbf{Nested Exact}} & 1024 & 23198 & 721916 \\
    \multicolumn{3}{l}{\textbf{Nested Recursive}} & 1024 & 23198 & 721916 \\
    \midrule
    \multirow{10}*{~\textbf{ML}}&\multirow{2}*{V-Cycle} & {Block-diagonal} & 924 & 2028 & 2420 \\
    &&{Super-node} & 508 & 1059 & 1508 \\
    & \multirow{2}*{W-Cycle} &{Block-diagonal} & 768 & 1926 & 2192 \\
    &&{Super-node} & 456 & 957 & 1432 \\
    & \multirow{2}*{F-Cycle} &{Block-diagonal} & 690 & 1722 & 1964 \\
    &&{Super-node} & 456 & 906 & 1356 \\
    & \multirow{2}*{Bottom-Up} &{Block-diagonal} & 1002 & 2436 & 2800 \\
    &&{Super-node} & 508 & 1161 & 1584 \\
    & \multirow{2}*{Even-Odd} &{Block-diagonal} & 1184 & 3099 & 3256 \\
    &&{Super-node} & 1028 & 2997 & 3028 \\
    \bottomrule
  \end{tabular*}
  \caption{Number of solved subsystems involving $B_i$ when solving the scenario tree NMPC problem for different tree depths and preconditioners.}
  \label{tab:runtimes}
\end{table}
Since the computational complexity varies among the preconditioners,
the performance of a preconditioner is not determined by its iteration
numbers alone. Hence, we analyze the running time of the preconditioners in
this section. It is important to note that in these experiments no
potential for parallelizability has been exploited as this is subject to future work.  
Additionally, the focus of the implementation is on the qualitative performance of the
preconditioners rather than optimizing for computational
efficiency. In order to provide insights into the running time behavior,
the number of solved subsystems involving $B_i$ (including the setup of the
preconditioner) is considered as a measure of running time, as the overall
computational burden is dominated by this. We consider
the scenario tree NMPC problem from Section~\ref{sec:scenario_mpc} for
the tree depths \num{5}, \num{10}, and
\num{15}. Table~\ref{tab:runtimes} shows the results for
preconditioners that converged in under \num{100} iterations. The
nested exact and nested recursive preconditioner have the same number
of solved subsystems. Even though these preconditioners lead to
convergence in one iteration, the computational complexity grows
exponentially with the tree depth, resulting in \num{721916} solved
subsystems. On the other hand, the ML preconditioners are
more robust with respect to the tree depth, only requiring between
\num{1356} and \num{3256} subsystems to be solved for the deepest
tree. Thus, the ML preconditioners are more suitable for
large-scale problems, while for shallow trees the difference between
the preconditioners becomes less significant.

\section{Conclusion}

We proposed and examined several algorithmic approaches to solve
saddle-point systems with a tree-based block structure. Apart from the
direct method, these approaches are based on preconditioned iterative
linear algebra. Several of the problem-specific preconditioners have
notable theoretical properties warranting their utilization in
large-scale problem instances.
This holds in particular for the ML methods obtained by
applying MG methods to tree-coupled systems. Much of the theory of
MG methods carries over to ML methods and the corresponding
algorithms work very well in our numerical experiments, with the more
accurate super-node smoothing combined with an F-cycle
iteration type being particularly efficient.

Valuable future work would include the study of the effects of
parallelization, the analysis of coupled systems with an even more
general graph-based structure including cyclic dependencies, and the
derivation of methods to automatically detect this exploitable
structure within given linear systems.

\section*{Acknowledgements}

CH was supported by the BMBF (Germany) grant 05M22VHA.
BH was supported by the MAC--MIGS Centre for Doctoral Training under the EPSRC (UK) grant EP/S023291/1.
JWP was supported by the EPSRC grant EP/S027785/1.

\setcounter{biburlnumpenalty}{3000}
\setcounter{biburlucpenalty}{6000}
\setcounter{biburllcpenalty}{9000}

\printbibliography

\appendix
\section{Domain Decomposition in the Mixed Formulation}
\label{app:dom_decomp}

In this section, we give a more detailed description of how domain decomposition can be applied to elliptic PDEs in the presented framework. We consider the Poisson problem on a given Lipschitz domain $\Omega$. The PDE is given by
\begin{align}
\begin{aligned}
    -\Delta u(x) &= f(x) &&\text{in } x \in \Omega, \\
    u(x) &= u_0(x) &&\text{on } x \in \Gamma_D, \\
    \frac{\partial u}{\partial n}(x) &= g(x) &&\text{on } x \in \Gamma_N,
\end{aligned}
    \label{eq:Poisson}
\end{align}
where $\Omega \subset \Real^d$ with $d \in \Nat$, $\partial \Omega = \Gamma_D \cup \Gamma_N$, $\Gamma_D \cap \Gamma_N = \emptyset$, $f \in L^2(\Omega)$, $u_0 \in H^{1/2}(\Gamma_D)$, and $g \in H^{-1/2}(\Gamma_N)$. Our aim is to solve the PDE on separate domains in order to break down the overall problem into smaller pieces. For that purpose, we decompose $\Omega$ into non-overlapping closed subsets $\Omega_i \subset \Omega$ for $i \in \{1, \dots, n\}$ for some $n \in \Nat$, known as non-overlapping domain decomposition (see~\cite{domain_decomposition}). We then have that $\Omega = \bigcup_{i=1}^n \Omega_i$ and $\Omega_i \cap \Omega_j$ is a set of measure zero for $i \neq j$. We denote the interfaces between the subdomains by $\Gamma_{i,j} \define \partial \Omega_i \cap \partial \Omega_j$ for $i, j \in \{1, \dots, n\}$ with $i \neq j$. One can then solve the PDE on each subdomain separately, while enforcing continuity of $u$ and of the flux on the interfaces, i.e.,
\begin{align*}
\begin{aligned}
u_i(x) &= u_j(x) &&\text{on } x \in \Gamma_{i,j}, \\
\nabla u_i(x) \cdot n &= -\nabla u_j(x) \cdot n &&\text{on } x \in \Gamma_{i,j},
\end{aligned}
\end{align*}
where $u_i$ denotes the solution on the subdomain $\Omega_i$. If the continuity at the interfaces is fulfilled while all $u_i$ solve the PDE~\eqref{eq:Poisson} on each subdomain, the overall solution can be achieved by joining the partial solutions $u_i$. While this visualizes the basic idea of non-overlapping domain decomposition, we choose to use the mixed formulation of the Poisson problem instead of \eqref{eq:Poisson}, since this formulation allows us to impose the continuity conditions through matching of degrees of freedom in magnitude (given an appropriate choice of approximation function spaces). One further advantage of using the mixed formulation is that it allows lower regularity of the data and solution. The mixed formulation is given by
\begin{equation}
  \begin{aligned}
    \sigma - \nabla u &= 0 &&\text{in } \Omega, \\
    \nabla \cdot \sigma &= -f &&\text{in } \Omega, \\
    \sigma \cdot n &= g &&\text{on } \Gamma_N, \\
    u &= u_0 &&\text{on } \Gamma_D.
  \end{aligned}
  \label{eq:mixed_formulation}
\end{equation}
Under appropriate assumptions, it can be shown that \eqref{eq:mixed_formulation} is equivalent to \eqref{eq:Poisson}. We use finite element methods to solve the PDE, which requires us to consider the weak formulation of the PDE:
\begin{align*}
    \int_\Omega \left(\tau \cdot \sigma + \left(\nabla \cdot \tau\right) u + v \left( \nabla \cdot \sigma \right) \right)\mathrm{d}x = -\int_\Omega f v \mathrm{d}x + \int_{\Gamma_D} \left(\tau \cdot n\right) u_0 \mathrm{d}s \\\text{ for all } (\tau, v) \in \Sigma \times U,
\end{align*}
where $(\sigma, u) \in \Sigma \times U$, $\Sigma = \{\sigma \in H(\operatorname{div}, \Omega) \mid \sigma \cdot n = g \text{ on } \Gamma_N \}$ and 
$U = L^2(\Omega)$.
Now, the continuity of the flux at the interface $\Gamma_{i,j}$ translates to 
\begin{align*}
    \sigma_i \cdot n = -\sigma_j \cdot n \quad\text{on } \Gamma_{i,j}.
\end{align*}
In order to give this meaning, we have to consider a weak formulation. One might be tempted to choose
\begin{align*}
    \int_{\Gamma_{i,j}} \sigma_i \cdot n \mathrm{d}s = -\int_{\Gamma_{i,j}} \sigma_j \cdot n \mathrm{d}s \quad \text{for all } v \in L^2(\Omega).
\end{align*}
However, one has to be aware that $v$ is not well-defined on $\Gamma_{i,j}$ and the trace of $\sigma$ is an operator $\gamma_{\Sigma, i,j} \in H^{-1/2}(\Gamma_{i,j})$. In \cite{poisson_dom_decomp}, it is discussed how this problem can be circumvented. They introduce an additional test function for the interfaces. First, the following spaces are introduced:
\begin{align*}
M &\define \prod_i L^2(\Gamma_i), \\
\Sigma &\define \{q \in \left(L^2(\Omega)\right)^d \mid \left.q\right|_{\Omega_i} \in H(\operatorname{div}, \Omega_i), [q \cdot n] \in M\}, \\
U &\define L^2(\Omega).
\end{align*}
Let $\Gamma^I$ denote the set of all interfaces. Let $u, v \in U$, $\sigma, \tau \in \Sigma$, and $\phi, \psi \in M$. The mixed formulation is then given by
\begin{equation}
\begin{aligned}
    \int_\Omega \left( \sigma \cdot \tau + v \nabla \cdot \sigma + u \nabla \cdot 
    \tau \right) \mathrm{d} x 
    + \int_{\Gamma^I} [\sigma \cdot n] \psi \mathrm{d} s - \int_{\Gamma^I} [\tau \cdot n] \phi \mathrm{d} s \\
    = -\int_\Omega f v \mathrm{d}x \quad \text{for all } v \in U, \tau \in \Sigma, \psi \in M.
\end{aligned}
\label{eq:mixed_formulation_lifted}
\end{equation}
Boundary terms have been omitted for the sake of simplicity. Another way to look at this is to enforce the continuity of the flux in a weak sense, but with trial and test functions that have higher regularity. The higher regularity of $\Sigma$ at the interfaces increases the regularity of the trace from $H^{-1/2}(\Gamma_{i,j})$ to $L^2(\Gamma_{i,j})$. This formulation enforces both the continuity of $u_i$ as well as the continuity of the flux at the interfaces, and yields matrices fitting the framework presented in this paper.

In \cite[Sec. 5]{poisson_dom_decomp}, conditions for conforming subspaces are stated under which convergence of solutions of \eqref{eq:mixed_formulation_lifted} to solutions to \eqref{eq:mixed_formulation} is guaranteed. We make use of the more specific case discussed in \cite[Eq. (5.16)]{poisson_dom_decomp}. It states that the space spanned by the traces of the discretization $\Sigma_h$ of $\Sigma$ at the interfaces has to be the same space as the discretization $M_h$ of $M$. The conforming subspaces
\begin{align*}
    M_h &\define \prod_i P_0(\Gamma_{i, h}) \subset M, \\
    \Sigma_h &\define RT_1(\Omega_h) \subset \Sigma, \\
    U_h &\define P_0(\Omega_h),
\end{align*}
fulfill this condition, where $RT_1(\Omega_h)$ denotes the Raviart--Thomas finite element function space of first order. The traces of $RT_1(\Omega_h)$ are constant at each facet, which spans the space $M_h$. This also shows that the traces of $\Sigma_h$ are in $L^2$. Moreover, we have that
\begin{align*}
    \operatorname{div} \Sigma_h \subseteq U_h.
\end{align*}

If the decomposition is carried out recursively one can arrange the continuity constraints in a tree form, as described in the numerical experiments.

\end{document}